\documentclass[a4paper]{amsart}
\usepackage{hyperref}   
\usepackage{amssymb,xcolor}
\usepackage{stmaryrd} 
\usepackage{enumerate}
\newcommand{\unfoldedcomment}[7]{}
\newcommand{\comment}[1]{}
\newcommand{\foldedcomment}[7]{}

\usepackage{mathtools}  
\mathtoolsset{showonlyrefs}

\newtheorem{theorem}{Theorem}[section]
\newtheorem{definition}[theorem]{Definition}
\newtheorem{lemma}[theorem]{Lemma}
\newtheorem{remark}[theorem]{Remark}
\newtheorem{proposition}[theorem]{Proposition}
\numberwithin{equation}{section}
\numberwithin{figure}{section}
\begin{document}

\title{Some rigidity theorems for spectral curvature bounds}

\author{Xiaoxiang Chai}
\address{School of Mathematics and Statistics, Central China Normal University,
  Wuhan, 430079 P.R. China}
\address{Department of Mathematics, POSTECH, 77 Cheongam-Ro, Nam-Gu, Pohang, Gyeongbuk, Korea 37673}
\email{ xxchai@kias.re.kr }

\author{Yukai Sun}
\address{School of Mathematics and Statistics, Henan University, Kaifeng 475004 P. R. China and Center for Applied Mathematics of Henan Province, Henan University, Zhengzhou 450046 P. R. China
}
\email{sunyukai@henu.edu.cn}

\begin{abstract}
  We investigate the geometric implications of spectral curvature
  bounds, extending classical rigidity results in scalar curvature geometry
  to the spectral setting. By systematically employing the warped
  $\mu$-bubble method, we show
  classification theorems for stable weighted minimal hypersurfaces in
  3-manifolds with nonnegative spectral scalar curvature, and we establish band width estimates for both spectral Ricci and spectral scalar curvatures. Furthermore, we prove some splitting theorems under spectral curvature conditions, including a spectral
  version of the Geroch conjecture for manifolds with arbitrary ends and a
  result related to the Milnor conjecture. 
\end{abstract}


\maketitle

\section{Introduction}

A celebrated result in scalar curvature geometry is the resolution of
the Geroch conjecture due to Schoen-Yau \cite{schoen-existence-1979}. The conjecture states that an
\(n\)-dimensional torus does not admit a metric of non-negative scalar
curvature. Schoen-Yau used a minimal hypersurface approach and it has 
become a major tool, and there are other approaches using
spinors \cite{gromov-positive-1983} and the technique of harmonic functions of Stern \cite{stern-scalar-2022}.

A weaker version of curvature, called the \textit{spectral
curvature}, which is defined as the first eigenvalue of an elliptic
operator involving the Laplace-Beltrami operator and the
curvature, recently has found its place in several important
problems. For example, the spectral curvature is useful in the stable Bernstein conjecture
\cite{chodosh-stable-2023}, \cite{chodosh-stable-2024}, \cite{mazet-stable-2024}
and the aspherical
conjecture \cite{chodosh-generalized-2024}. The earliest occurrence of such a notion seems to be in
\cite{schoen-existence-1983}.

We will use the following definition of the spectral curvatures using a positive function.

\begin{definition}
  \label{def spec}Let $(M, g)$ be a Riemannian manifold and $u$ be a positive
  function, we call
  \begin{equation}
    - \gamma u^{- 1} \Delta_g u + \tfrac{1}{2} R_g \label{eq spec scalar}
  \end{equation}
  the spectral scalar curvature and
  \begin{equation} \label{eq spec ricci1}
    - \gamma u^{- 1} \Delta_g u + \operatorname{Ric}_g 
  \end{equation}
  the spectral Ricci curvature where
  \begin{equation}
    \operatorname{Ric}_g := \inf_{e \in T_x M, |e|_g = 1} \operatorname{Ric}_g (e, e)
    \label{eq Rc}
  \end{equation}
  is the least Ricci curvature at a given point $x \in M$. We will say
  spectral scalar (resp. Ricci) curvature modified by $\gamma$ and $u$ if the
  dependence on $\gamma$ and $u$ needs to be explicit.
\end{definition}

As easily checked, for a closed manifold, a lower bound on \eqref{eq spec scalar} would imply
the same lower bound on the first eigenvalue of \( -\gamma \Delta_g + \tfrac{1}{2} R_g \) (see \cite[Theorem 1]{fischer-colbrie-structure-1980}). Similar implications
work for \eqref{eq spec ricci1}.

Here we fix the convention of $\ensuremath{\operatorname{Ric}}_g$: if it is
given no argument it means the least Ricci curvature; two arguments mean the
usual Ricci curvature. We also find it convenient to use
\begin{equation}
  - \gamma u^{- 1} \Delta_g u + 2\ensuremath{\operatorname{Ric}}_g . \label{eq
  ric2 spec}
\end{equation}

\subsection{Weighted minimal hypersurfaces}

A suitable generalization of Schoen-Yau's technique of minimal
hypersurfaces in the study of spectral curvatures, such as the
spectral Ricci curvature and the spectral scalar curvature, is the
notion of a weighted minimal hypersurface, which arises as a critical
point of the weighted area functional.

\begin{definition}
  \label{def weighted minimal}We say that $\Sigma$ is a $u^{\gamma}$-weighted
  minimal hypersurface if it is a critical point of the $u^{\gamma}$-weighted
  area functional
  \begin{equation}
    \mathcal{A}_u (S) = \int_S u^{\gamma} \mathrm{d} \mathcal{H}^{n - 1}
    \label{weighted area functional}
  \end{equation}
  defined for all oriented hypersurfaces. The references to $u^{\gamma}$ would
  be omitted if the dependence on $\gamma$ and $u$ are clear.
\end{definition}

Given a smooth family of hypersurfaces $\{\Sigma_t \}$ such that $\Sigma_0 =
\Sigma$, the first variation of the weighted area functional is
\[ \tfrac{\mathrm{d}}{\mathrm{d} t} \mathcal{A}_u (\Sigma_t) |_{t = 0} =
   \int_{\Sigma} (H + \gamma u^{- 1} u_{\nu}) u^{\gamma} \langle Y, \nu
   \rangle \mathrm{d} \mathcal{H}^{n - 1}, \]
where $\nu$ is a chosen unit normal to $\Sigma$, $H$ is the mean curvature
defined as $\ensuremath{\operatorname{div}}_S \nu$ and $Y$ is the variational
vector field.
\begin{definition}
We call $H + \gamma u^{- 1} u_{\nu}$ the
$u^{\gamma}$-\text{{\itshape{weighted}}} \text{{\itshape{mean}}}
\text{{\itshape{curvature}}}. If the second variation is non-negative for a
weighted minimal hypersurface, we call $\Sigma$ \text{{\itshape{stable}}}, and
we call $\Sigma$ $u^{\gamma}$-\text{{\itshape{weighted}}}
\text{{\itshape{area}}}-\text{{\itshape{minimizing}}} if $\Sigma$ is a
minimizer to the functional \eqref{weighted area
  functional}. Again, the references to $u^{\gamma}$ would be
omitted if the dependence on $\gamma$ and $u$ are clear.
\end{definition}

We will introduce
generalizations of the weighted minimal hypersurfaces in Section \ref{sec warp}, in particular,
warped $h$-bubbles.

We now state our first result regarding the weighted minimal hypersurface, which is
a classification result of stable weighted surfaces in a 3-manifold of non-negative
spectral curvature. The result is an analog of {\cite[Theorem
3]{fischer-colbrie-structure-1980}}, {\cite{cai-rigidity-2000}} which
classifies stable minimal surfaces in a 3-manifold of non-negative scalar
curvature.

\begin{theorem}
  \label{thm weighted fcs}Let \( 0 \leqslant \gamma <4 \),
  $(M^3, g)$ be a 3-dimensional complete manifold
  with spectral nonnegative scalar curvature and $\Sigma$ be a stable,
  complete, oriented weighted minimal surface with weight $u^{\gamma}$ in $(M,
  g)$. Then there are two possibilities:
\begin{enumerate}[(a)]
    \item \label{weight fcs a} $\Sigma$ is compact, then $\Sigma$ is a sphere or a torus, in the
    case of a torus, $\Sigma$ is flat, if further $\Sigma$ is weighted
    area-minimizing, then $(M, g)$ locally splits;
    
  \item \label{weight fcs b} $\Sigma$ is non-compact and \( 0 \leqslant \gamma < 3 \), then $\Sigma$ is conformally equivalent to
    the complex plane $\mathbb{C}$ or a cylinder $\mathbb{A}$; $\Sigma$ is
    flat if it is conformally equivalent to a cylinder.
  \end{enumerate}
\end{theorem}
\begin{remark}\label{range gamma}
The case with $\Sigma$ being non-compact is more subtle than {\cite[Theorem
3]{fischer-colbrie-structure-1980}} in Theorem \ref{thm
weighted fcs} and subsequent Theorem \ref{thm fcs ric2}. It is due to a question posed in
{\cite[Remark 1]{fischer-colbrie-structure-1980}} regarding the inverse
spectral properties of the elliptic operator $- \Delta_{\Sigma} + a K$ where
$a \in \mathbb{R}$ and $K$ is the Gaussian curvature of the surface $\Sigma$.
Here, we have made use of {\cite{berard-inverse-2014}}. Because of this, we need the condition
\( 0 \leqslant \gamma <3 \); and 
an extra assumption on the volume growth of \( \Sigma \) is required when \( 3 \leqslant \gamma <4 \), the case
which we do not discuss in Theorem \ref{thm weighted fcs}.
\end{remark}

The strategy of proving the compact case of Theorem \ref{thm weighted fcs} is
as follows: we use the stability and the spectral curvature condition (with
the Gauss-Bonnet theorem) to show that $\Sigma$ is infinitesimally rigid (cf.
{\cite{fischer-colbrie-structure-1980}}); second, construct a foliation
$\{\Sigma_t \}$ with $\Sigma$ be a leaf using Theorem \ref{lm foliation};
third, determine the sign of the weighted mean curvature of each leaf using the
curvature condition again; fourth, show that every nearby $\Sigma_t$ is also
minimizing using the first variation (cf. {\cite{bray-rigidity-2010}}). We can
show a global splitting if we assume an additional topological condition, see
Theorem \ref{thm cylinder splitting} (cf. {\cite{cai-rigidity-2000}}).

This strategy for proving the rigidity will be used in Subsection \ref{subsec
cwmc} and the band width estimates in Subsection \ref{subsec intro bw} with
suitable adaptations. To avoid repetition, we unify and streamline these proofs
by introducing the variables $Z$ and $W$ in Subsection \ref{sec rewrite}. The
key differences left are estimating the mean
curvature of a leaf of the foliation and calculating the rigid metrics.

With a slightly different proof, we have the following.

\begin{theorem}
  \label{thm fcs ric2}Theorem \ref{thm weighted fcs} (\ref{weight fcs a}) holds if the spectral
  non-negative scalar curvature is replaced by the condition that $- \gamma
  u^{- 1} \Delta_g u + 2\ensuremath{\operatorname{Ric}}_{g} \geqslant 0$ with $0
  \leqslant \gamma < 4$ and Theorem \ref{thm weighted fcs} (\ref{weight fcs b}) holds if additionally \( 0 \leqslant \gamma < \tfrac{7}{2} \).
\end{theorem}
\begin{remark}
  Again, for the range of \( \gamma \), see Remark \ref{range gamma} and \cite{berard-inverse-2014}.
  \end{remark}

For the Ricci curvature $- \gamma u^{- 1} \Delta_g u
+\ensuremath{\operatorname{Ric}}_{g}$, the splitting result was obtained by
{\cite{antonelli-sharp-2024,catino-criticality-2025}} which works in any
dimensions.

\subsection{Constant weighted mean curvature}\label{subsec cwmc}

The concept of a (Riemannian) \text{{\itshape{band}}} is introduced by Gromov
in his studies of metric inequalities {\cite{G2018}}. It is a Riemannian
manifold with at least two boundaries. We can group the boundaries into two
groups $\partial_- M$ and $\partial_+ M$ allowing the \text{{\itshape{band
width}}} to be defined as the distance from $\partial_- M$ to $\partial_+ M$.

We fix the notation of band, let $\Sigma$ be an oriented hypersurface
homologous to $\partial_- M$, we choose the direction of the unit normal of
$\Sigma$ such that it points outward from the region bounded by $\Sigma$ and
$\partial_- M$ and we denote it by $\nu_\Sigma$. In particular, we choose the direction of the unit normal \( \nu_{-} \)
of
$\partial_- M$ pointing to the inside of $M$ and the unit normal \( \nu_{+} \) of $\partial_+ M$ outside of $M$.

By using hypersurfaces of constant weighted mean curvature, we can show some
rigidity theorems for 
\( n \)-dimensional \textit{torical bands} which are simply \( T^{n-1}\times [-1,1] \) with a smooth metric \( g \).
Here, \( T^n \) denotes the \( n \)-dimensional torus.


Below is a spectral analog of {\cite[Theorem 1.1]{andersson-rigidity-2008}}.
\comment{
\begin{theorem}
  \label{thm acg}Let $0 \leqslant \gamma < \tfrac{2 n}{n - 1}$, $\Lambda < 0$
  and
  \[ \beta = \frac{\sqrt{- 2 \Lambda}}{\sqrt{2 (n - 1) - (n - 2) \gamma} 
     \sqrt{2 n - (n - 1) \gamma}}, \text{ } \alpha = (2 - \gamma) \beta . \]
 Let $(M,g)$ be an \( n \)-dimensional torical band
 and $u > 0$ such that
  \begin{equation}
    - \gamma u^{- 1} \Delta_g u + \tfrac{1}{2} R_g \geqslant \Lambda,
    \label{neg eigen sc}
  \end{equation}
  and
  \[ H_+ + \gamma u^{- 1} u_{\nu_+} \geqslant (n - 1) \alpha + \gamma \beta
     \text{ along } \partial_+ M =\{1\} \times T^{n - 1}, \]
  and
  \[ H_- + \gamma u^{- 1} u_{\nu_-} \leqslant (n - 1) \alpha + \gamma \beta
     \text{ along } \partial_- M =\{- 1\} \times T^{n - 1} . \]
  Then $(M, g)$ must be isometric to $([t_-, t_+] \times T^{n - 1}, \mathrm{d}
  t^2 + e^{2 \alpha t} g_{\mathbb{T}^{n - 1}})$ for some $t_- < t_+$ with $u$
  given by a constant multiple of $e^{\beta t}$.
\end{theorem}
}
\begin{theorem}
  \label{thm acg}Let $0 \leqslant \gamma < \tfrac{2 (n - 1)}{n - 2}$, $\Gamma
  = \tfrac{2 n - (n - 1) \gamma}{4 (n - 1) - 2 (n - 2) \gamma}$, and $\Lambda$
  be a constant such that either $\Gamma \Lambda < 0$ or $\Gamma = \Lambda =
  0$. If $\Gamma \Lambda < 0$, set $\eta = \eta_{\Gamma, \Lambda} : = \sqrt{-
  \Lambda / \Gamma}$; and if $\Gamma = \Lambda = 0$, set $\eta = 0$. Set
  $\beta = \tfrac{1}{2 (n - 1) - (n - 2) \gamma} \eta$ and $\alpha = \tfrac{2
  - \gamma}{2 (n - 1) - (n - 2) \gamma} \eta$. If $(M, g)$ is an
  $n$-dimensional torical band and $u > 0$ such that
  \begin{equation}
    - \gamma u^{- 1} \Delta_g u + \tfrac{1}{2} R_g \geqslant \Lambda
    \label{neg eigen sc}
  \end{equation}
  in \( M \)
  and
  \[ H_{\partial_+M} + \gamma u^{- 1} u_{\nu_+} \geqslant \eta \text{ }  \text{along }
     \partial_{+} M =\{1\} \times T^{n - 1}, \]
  and
  \[ H_{\partial_-M} + \gamma u^{- 1} u_{\nu_-} \leqslant \eta \text{ }  \text{along }
     \partial_- M =\{- 1\} \times T^{n - 1} . \]
  Then $(M, g)$ must be isometric to $([t_-, t_+] \times T^{n - 1}, \mathrm{d}
  t^2 + e^{2 \alpha t} g_{T^{n - 1}})$ for some $t_- < t_+$ with $u$
  given by a constant multiple of $e^{\beta t}$. Here, $g_{T^{n - 1}}$ is some flat metric on $T^{n - 1}$
\end{theorem}


As for the spectral Ricci curvatures, we have the following.

\begin{theorem}
  \label{thm acg ric2}Let $0 \leqslant \gamma < 3 + \tfrac{1}{n - 2}$,
  $\Lambda \leqslant 0$, $\eta = \sqrt{- \Lambda} / \sqrt{1 - \gamma / 4}$,
  $(M, g)$ be an \( n \)-dimensional torical
  band and $u > 0$ such that
  \begin{equation}
    - \gamma u^{- 1} \Delta_g u + (n - 1) \ensuremath{\operatorname{Ric}}_{g}
    \geqslant \Lambda, \label{neg eigen ric n-1}
  \end{equation}
  in \( M \)
  and
  \[ H_{\partial_+M} + \gamma u^{- 1} u_{\nu_+} \geqslant \eta \text{ along } \partial_+ M
     =\{1\} \times T^{n-1}, \]
  and
  \[ H_{\partial_-M} + \gamma u^{- 1} u_{\nu_-} \leqslant \eta \text{ along } \partial_- M
     =\{- 1\} \times T^{n-1} . \]
  Then $(M, g)$ is foliated by flat $(n - 1)$-torus with $H + \gamma u^{- 1}
  u_{\nu} = \eta$ and the equality in \eqref{neg eigen ric n-1} is achieved
  everywhere. In particular, if $n = 3$, then $(M, g)$ is isometric to
  \[ ([t_-, t_+] \times T^2, \mathrm{d} t^2 {+ e^{\eta (1 - \gamma / 2) t}} 
     g_{T^2}) \]
  for some $t_- < t_+$ with $u$ given by a constant multiple of $e^{\eta t /
  2}$.
\end{theorem}

\begin{remark}
  \label{rk acg ric2}For $n \geqslant 4$, $u = e^{\eta t / 2}$ and the warped
  product metric $g = \mathrm{d} t^2 + e^{2 (1 - \gamma / 2) / (n - 1) t}
  g_{\mathbb{T}^{n - 1}}$ satisfy the conditions in Theorem \ref{thm acg
  ric2}, however, we do not know whether there are other metrics.
\end{remark}
\begin{theorem}
\label{thm acg ric1}Let $0 \leqslant \gamma < \tfrac{n - 1}{n - 2}$, $\Gamma =
\tfrac{4 - (n - 1) \gamma}{4 (n - 1 - (n - 2) \gamma)}$, $\Lambda$ be a number
such that $\Gamma \Lambda < 0$ or $\Gamma = \Lambda = 0$. Let $\eta$ be the
constant $\sqrt{- \Lambda / \Gamma}$ when $\Gamma \Lambda < 0$ and $0$ when
$\Gamma = \Lambda = 0$. Let $\beta = - \tfrac{n - 3}{2 (n - 1 - (n - 2)
\gamma)}\eta$ and $\alpha = \tfrac{2 - \gamma}{2 (n - 1 - (n - 2) \gamma)}\eta$. If
($M, g$) is a band with
\begin{enumerate}[(a)]
  \item $- \gamma u^{- 1} \Delta_g u + \mathrm{Ric}_g
    \geqslant \Lambda$ in \( M \),\\
  \item \( H_{\partial_+M} + \gamma u^{- 1} u_{\nu_+} \geqslant \eta \text{ }  \text{along }
     \partial_{+} M , \)
  and
  \( H_{\partial_-M} + \gamma u^{- 1} u_{\nu_-} \leqslant \eta \text{ }  \text{along }
     \partial_- M  \),
    \end{enumerate}
    then ($M, g$) is isometric to ($[t_-, t_+] \times S, \mathrm{d} t^2 + e^{2 \alpha
t} g_S$) for some $t_- < t_+$ and ($S, g_S$) with $\mathrm{Ric}_{g_S}
\geqslant 0$ and $u$ is a constant multiple of $e^{\beta t}$.
  \end{theorem}
\comment{
\begin{theorem}
  \label{thm acg ric1}Let $0 \leqslant \gamma < \tfrac{4}{n - 1}$, $\Lambda
  \leqslant 0$,
  \[ \alpha = \tfrac{2 - \gamma}{\sqrt{4 - (n - 1) \gamma}}, \text{ } \beta =
     \tfrac{(3 - n) \gamma}{\sqrt{n - 1 - (n - 2) \gamma} \sqrt{4 - (n - 1)
     \gamma}}, \text{ } \eta = \tfrac{2 \sqrt{n - 1 - (n - 2) \gamma}}{\sqrt{4
     - (n - 1) \gamma}} . \]
  Let $(M, g)$ be a band with $- \gamma u^{- 1} \Delta_g u
  +\ensuremath{\operatorname{Ric}}_g \geqslant \Lambda$, $H_{\partial_+ M} +
  \gamma u^{- 1} u_{\nu} \geqslant \eta$ (resp. $\leqslant \eta$) along
  $\partial_+ M$ (resp. $\partial_- M$), then $(M, g)$ is isometric to $([t_-,
  t_+] \times S, \mathrm{d} t^2 + e^{2 \alpha t} g_S)$ for some $t_- < t_+$
  and $(S, g_S)$ with $\ensuremath{\operatorname{Ric}}_{g_S} \geqslant 0$ and
  $u$ is a constant multiple of $e^{\beta t}$.
\end{theorem}
}
\subsection{Band width estimates}\label{subsec intro bw}

One of the basic results of the band width is that the band width of a torical
band is bounded from above due to the effect
of the positive scalar curvature, see \cite{G2018}. Gromov's approach is by
considering hypersurfaces of prescribed mean curvature \( h \) where \( h \)
is a Lipschitz function related to the band width.
\comment{
To facilitate the description of the band width estimates, we introduce some
notations: for a given positive constant $\Gamma$, we fix $\sigma$ to be the
sign of $\Lambda$,
\begin{equation}
  B_{\Lambda, \Gamma} = \left\{\begin{array}{ll}
    \sqrt{| \Lambda | \Gamma} & \Lambda \neq 0,\\
    1, & \Lambda = 0,
  \end{array}\right. \label{eq B}
\end{equation}
and
\begin{equation}
  \label{eq eta gamma lambda} \eta_{\Lambda, \Gamma} (t) =
  \left\{\begin{array}{ll}
    \sqrt{- \Lambda / \Gamma} \coth (\sqrt{- \Lambda \Gamma} t), & \Lambda <
    0,\\
    \tfrac{1}{\Gamma t}, & \Lambda = 0,\\
    \sqrt{\Lambda / \Gamma} \cot (\sqrt{\Lambda \Gamma} t), & \Lambda > 0.
  \end{array}\right.
\end{equation}
It is clear that $\eta = \eta_{\Gamma, \Lambda}$ solves the ODE
\begin{equation}
  \Gamma \eta^2 + \eta' + \Lambda = 0 \text{ and } \eta' < 0. \label{eq
  general ode}
\end{equation}
We also define $\ensuremath{\operatorname{sn}}_{\sigma}$ by the following
\[ \ensuremath{\operatorname{sn}}_{\sigma} (t) = \left\{\begin{array}{ll}
     \sin t, & \text{ if } \sigma> 0,\\
     t, & \text{ if } \sigma = 0,\\
     \sinh t, & \text{ if } \sigma < 0.
   \end{array}\right. \]
The following spectral version of the positive scalar curvature can also put a
constraint on the band width, see {\cite{chai-band-2025}}. Here, we give a slight
generalization of {\cite{chai-band-2025}} by allowing negative and zero spectral
scalar curvature bounds.
}
The proper generalization of Gromov's \( h \)-bubble is the notion of a \textit{warped} \( h \)-\textit{bubble} which we will use to show several band
width estimates under spectral curvature bounds.

To facilitate the description of the band width estimates, we introduce some
notations. Let $\Gamma$ and $\Lambda$ be two constants and we are concerned
with the ODE
\begin{equation}
  \Gamma \eta^2 + \eta' + \Lambda = 0 \label{eq general ode}
\end{equation}
such that the solution $\eta$ satisfies $\eta' < 0$. To ensure $\eta' < 0$, at least one of $\Gamma$ and
$\Lambda$ should be positive. Indeed, the solution to \eqref{eq general ode}
is given by the following
\begin{equation}
  \label{eq eta gamma lambda} \eta(t):=\eta_{\Lambda, \Gamma} (t) := \left\{
  \begin{array}{ll}
    \sqrt{- \Lambda / \Gamma} \coth (\sqrt{- \Lambda \Gamma} t), & \Gamma > 0,
    \text{ } \Lambda < 0,\\
    \frac{1}{\Gamma t}, & \Gamma > 0, \text{ } \Lambda = 0,\\
    \sqrt{\Lambda / \Gamma} \cot (\sqrt{\Lambda \Gamma} t), & \Gamma > 0,
    \text{ } \Lambda > 0,\\
    - \Lambda t, & \Gamma = 0, \text{ } \Lambda > 0,\\
    - \sqrt{- \Lambda / \Gamma} + \frac{2 \sqrt{- \Lambda / \Gamma}}{1 + \exp (2
    \sqrt{- \Lambda \Gamma} t)} & \Gamma < 0, \text{ } \Lambda > 0.
  \end{array} \right.
\end{equation}
Evidently, $\eta_{\Lambda, \Gamma}$ is only well defined on the interval
$I_{\Lambda, \Gamma}$ given by
\begin{equation}
  \label{interval} I_{\Lambda, \Gamma} := \left\{ \begin{array}{ll}
    (0, \infty), & \Gamma > 0, \text{ } \Lambda \leqslant 0 ;\\
    (0, \pi / \sqrt{\Lambda \Gamma}), & \Gamma > 0, \text{ } \Lambda > 0,\\
    (- \infty, + \infty), & \Gamma \leqslant 0, \text{ } \Lambda > 0.
  \end{array} \right.
\end{equation}
Now we give a slight generalization of the band width estimate under a
positive spectral scalar curvature bound in {\cite{chai-band-2025}} by
allowing negative and zero spectral scalar curvature bounds.
\comment{
\begin{theorem}
  \label{thm bw spec scalar}Let $\Lambda$ be a constant, $0 \leqslant \gamma <
  \tfrac{2 n}{n - 1}$,
  \[ \alpha = \tfrac{2 (2 - \gamma)}{2 n - (n - 1) \gamma}, \text{ } \beta =
     \tfrac{2}{2 n - (n - 1) \gamma}, \text{ } \Gamma = \tfrac{2 n - (n - 1)
     \gamma}{4 (n - 1) - (n - 2) \gamma}, \]
  $B = B_{\Lambda, \Gamma}$ and $\eta = \eta_{\Lambda, \Gamma}$. Let $(M^n,
  g)$ be a torical band such that
\begin{enumerate}[(a)]
    \item there exists a positive function $u$ with $- \gamma u^{- 1} \Delta_g
    u + \tfrac{1}{2} R_g \geqslant \Lambda$;
    
    \item $H_{\partial_+ M} + \gamma u^{- 1} u_{\nu_+} \geqslant \eta (t_+) $
    on $\partial_+ M$, $H_{\partial_- M} + \gamma u^{- 1} u_{\nu_-} \leqslant
    \eta (t_-)$ for some $0 < t_- < t_+$, and if $\Lambda > 0$, we also
    require that $t_+ < B^{- 1} \pi$,
  \end{enumerate}
  then
  \[ \operatorname{width} (M, g) \leqslant t_+ - t_- . \]
  Equality occurs if and only if $(M, g)$ is isometric to the model
  \[ ([t_-, t_+] \times T^{n - 1}, \mathrm{d} t^2 + \phi (t)^2
     g_{\mathbb{T}^{n - 1}}) \]
  where $\phi (t) =\ensuremath{\operatorname{sn}}_{\sigma} (B t)^{\alpha}$ for
  some flat metric $g_{\mathbb{T}^{n - 1}}$ on $T^{n - 1}$ and $u$ is a
  constant multiple of $\ensuremath{\operatorname{sn}}_{\sigma} (B
  t)^{\beta}$.
\end{theorem}
}
\begin{theorem}
  \label{thm bw spec scalar}Let $\Lambda$ be a constant, $0 \leqslant \gamma <
  \tfrac{2 (n - 1)}{n - 2}$, $\Gamma = \tfrac{2 n - (n - 1) \gamma}{4 (n - 1)
    - 2(n - 2) \gamma}$. Assume that at least one of \( \Gamma \) and \( \Lambda \) is positive,
 and let $t_- < t_+$ be two
  numbers such that $[t_-, t_+] \subset I_{\Lambda, \Gamma}$. Let $(M^n, g)$
  be a torical band such that
  \begin{enumerate}[(a)]
    \item there exists a positive function $u$ with $- \gamma u^{- 1} \Delta_g
    u + \tfrac{1}{2} R_g \geqslant \Lambda$,
    
    \item and $H_{\partial_+ M} + \gamma u^{- 1} u_{\nu_+} \geqslant \eta
    (t_+)$ on $\partial_+ M$, $H_{\partial_- M} + \gamma u^{- 1} u_{\nu_-}
    \leqslant \eta (t_-)$ on $\partial_- M$,
  \end{enumerate}
  then
  \[ \operatorname{width} (M, g) \leqslant t_+ - t_- . \]
  Equality occurs if and only if $(M, g)$ is isometric to the model
  \[ ([t_-, t_+] \times T^{n - 1}, \mathrm{d} t^2 + \phi (t)^2
     g_{T^{n - 1}}) \]
  where $\phi (t) = \exp (\tfrac{2 - \gamma}{2 (n - 1) - (n - 2) \gamma}
  \int^t \eta)$ and $u$ is a constant multiple of $\exp (\tfrac{1}{2 (n - 1) - (n - 2)
  \gamma} \int^t \eta)$.
\end{theorem}
\begin{remark}
  A new feature of this band width estimate is the possible negative or zero value of \( \Gamma \) in the
  ODE \eqref{eq general ode}, which is not present in the non-spectral case (i.e. \( \gamma =0\)).
  \end{remark}
Now let's turn to the Ricci curvature. The Bonnet-Myers theorem is a fundamental result
regarding the control of diameter by a Ricci curvature bound. More specifically,
it states that if a complete manifold with
$\ensuremath{\operatorname{Ric}}_{g} \geqslant n - 1$, then its diameter must be
less than $\pi$. This can be interpreted as a band type estimate: any band
with $\ensuremath{\operatorname{Ric}}_{g} \geqslant n - 1$ must have its width
less than $\pi$. Indeed, for a closed manifold, we can remove a pair of points
which realizes the diameter, and the manifold becomes a band in some sense. This
interpretation was achieved by {\cite{hirsch-rigid-2025}} via spacetime
harmonic function techniques.

Now, we apply the same interpretation for the Bonnet-Myers theorem of spectral
Ricci curvature which was shown by Antonelli-Xu {\cite{AX2024}}. Their theorem
states that if a complete manifold has positive spectral Ricci curvature with
$0 \leqslant \gamma < 4 / (n - 1)$, then the diameter of the manifold is
bounded. This result has been further extended to the spectral Bakry-Emery
curvatures {\cite{chu-spectral-2024,yeung-spectral-2025,wu-spectral-2025}}.
With the alternative interpretation, we can show the band with estimates with
different curvature bounds and a rigidity statement.
\comment{
\begin{theorem}
  \label{bm bw}Let $\Lambda$ be a real constant, $0 \leqslant \gamma <
  \tfrac{4}{n - 1}$,
  \[ \alpha = \tfrac{2 (2 - \gamma)}{4 - (n - 1) \gamma}, \beta = - \tfrac{2
     (n - 3)}{4 - \gamma (n - 1)}, \text{ } \Gamma = \tfrac{4 - (n - 1)
     \gamma}{4 (n - 1 - (n - 2) \gamma)}, \]
  $B = B_{\Lambda, \Gamma}$ and $\eta = \eta_{\Lambda, \Gamma}$. Let $(M, g)$
  be a band such that
\begin{enumerate}[(a)]
    \item there exists a positive function $u$ with
    \begin{equation}
      - \gamma u^{- 1} \Delta_g u + \operatorname{Ric} \geqslant \Lambda ; \label{eq
      ric1 spec}
    \end{equation}
    \item $H_{\partial_+ M} + \gamma u^{- 1} u_{\nu_+} \geqslant \eta (t_+)$
    on $\partial_+ M$, $H_{\partial_- M} + \gamma u^{- 1} u_{\nu_-} \leqslant
    \eta (t_-)$ for some $0 < t_- < t_+$ and if $\Lambda > 0$, we require $t_+
    < B^{- 1} \pi$,
  \end{enumerate}
  then
  \[ \operatorname{width} (M, g) \leqslant t_+ - t_- . \]
  Equality occurs if and only if $(M, g)$ is isometric to the model
  \begin{equation}
    ([t_-, t_+] \times S, \mathrm{d} t^2 + \phi^2 (t) g_S) \label{eq ric1
    model}
  \end{equation}
  where $\phi (t) = \mathrm{sn}_{\sigma} (Bt)^{\alpha}$ for some closed
  manifold $(S, g_S)$ with its Ricci curvature
  \begin{equation}
    \operatorname{Ric}_{g_S} \geqslant - (n - 2) \max_{t \in [t_-, t_+]} \phi^2
    (\tfrac{\phi'}{\phi})', \label{eq ric1 level ricci}
  \end{equation}
  and $u$ is a constant multiple of $\mathrm{sn}_{\sigma} (Bt)^{\beta}$.
\end{theorem}
}
\begin{theorem}
  \label{bm bw}Let $\Lambda$ be a constant, $0 \leqslant \gamma < \tfrac{n -
  1}{n - 2}$ and $\Gamma = \tfrac{4 - (n - 1) \gamma}{4 (n - 1 - (n - 2) \gamma)}$.
  Assume that at least one of $\Gamma$ and $\Lambda$ is positive,
  and $t_- < t_+$ be two numbers such that $[t_-, t_+] \subset I_{\Lambda,
  \Gamma}$. Let $(M, g)$ be a band such that
  \begin{enumerate}[(a)]
    \item there exists a positive function $u$ with
    \begin{equation}
      - \gamma u^{- 1} \Delta_g u + \operatorname{Ric}_{g} \geqslant \Lambda ; \label{eq
      ric1 spec}
    \end{equation}
    \item $H_{\partial_+ M} + \gamma u^{- 1} u_{\nu_+} \geqslant \eta (t_+)$
    on $\partial_+ M$, $H_{\partial_- M} + \gamma u^{- 1} u_{\nu_-} \leqslant
    \eta (t_-)$ on $\partial_- M$,
  \end{enumerate}
  then
  \[ \operatorname{width} (M, g) \leqslant t_+ - t_- . \]
  Equality occurs if and only if $(M, g)$ is isometric to the model
  \begin{equation}
    ([t_-, t_+] \times S, \mathrm{d} t^2 + \phi^2 (t) g_S) \label{eq ric1
    model}
  \end{equation}
  where $\phi (t) = \exp (\tfrac{2 - \gamma}{2 (n - 1 - (n - 2) \gamma)}
  \int^t \eta)$ for some closed manifold $(S, g_S)$ with its Ricci curvature
  \begin{equation}
    \mathrm{Ric}_{g_S} \geqslant - (n - 2) \max_{t \in [t_-, t_+]} \phi^2
    (\tfrac{\phi'}{\phi})', \label{eq ric1 level ricci}
  \end{equation}
  and $u$ is a constant multiple of $\exp (-\tfrac{n-3}{2 (n - 1 - (n - 2)
  \gamma)} \int^t \eta)$.
\end{theorem}

\begin{remark}
  It should be viable to prove via minimizing weighted geodesics from
  $\partial_- M$ to $\partial_+ M$ (see
  {\cite{catino-criticality-2025,hong-splitting-2025}}). Our approach is based
  on warped $\mu$-bubbles and is consistent with other parts of the
  article.
\end{remark}

\begin{remark}
  The equality case of Theorem \ref{bm bw} suggests that the diameter estimate
  proved in {\cite{AX2024}} cannot achieve an equality for $\gamma > 0$,
  because the completeness requires that $S$ is an $(n - 1)$-dimensional
  sphere and $\alpha = 1$ (from which $\gamma = 0$ for $n > 3$).
\end{remark}




\begin{theorem}
  \label{thm bw ric2}Let $0 \leqslant \gamma < 3 + \tfrac{1}{n - 2}$, $\Gamma
  = 1 - \gamma / 4$ (note that \( \Gamma>0 \)), and \( [t_-,t_+] \subset I_{\Lambda,\Gamma} \).
 If a torical band $M = [- 1, 1] \times T^{n - 1}$ with a metric
  $g$ and some positive function $u$ satisfy
\begin{enumerate}[(a)]
    \item $- \gamma u^{- 1} \Delta_g u + (n - 1) \operatorname{Ric}_{g} \geqslant
    \Lambda$,
    
    \item $H_{\partial_+ M} + \gamma u^{- 1} u_{\nu_+} \geqslant \eta (t_+)$
    on $\partial_+ M$ and $H_{\partial_- M} + \gamma u^{- 1} u_{\nu_{-}} \leqslant
    \eta (t_-)$ on $\partial_- M$,
  \end{enumerate}
 then
  \begin{equation}
    \ensuremath{\operatorname{width}}(M, g) \leqslant t_+ - t_- . \label{eq bw
    ric n-1}
  \end{equation}
\end{theorem}

The rigidity of the band width estimate in Theorem \ref{thm bw ric2} is
interesting, and we are able to show the following result. Note that the
exponent $\gamma$ now has a smaller range and the dimension is $3$.

\begin{theorem}
  Let $0 \leqslant \gamma \leqslant 2$ and $(M, g)$ be a 3-dimensional torical
  band which satisfies the assumptions of Theorem \ref{thm bw ric2}, the
  equality in \eqref{eq bw ric n-1} is achieved if and only if the torical
  band is isometric to the model
  \[ \bar{g} = \mathrm{d} t^2 + \phi (t)^2 \mathrm{d} s_1^2 + \varphi (t)^2
     \mathrm{d} s_2^2, \]
  where $\phi$ and $\varphi$ are given by the following: If $\Lambda = 0$,
  \[ \phi' / \phi + \varphi' / \varphi = \tfrac{1 - \gamma / 2}{1 - \gamma /
     4} t^{- 1}, \phi' / \phi - \varphi' / \varphi = F_0 t^{- \tfrac{1 -
     \gamma / 2}{1 - \gamma / 4}} \]
  with $F_0 \in \mathbb{R}$, $t \in [t_-, t_+]$ satisfying $F_0^2 t^{\tfrac{2
  \gamma}{4 - \gamma}} \leqslant \tfrac{1 - \gamma / 2}{1 - \gamma / 4}$.
  
  If $\Lambda < 0$,
\begin{align}
\phi' / \phi + \varphi' / \varphi & = (1 - \gamma / 2) \sqrt{\tfrac{-
\Lambda}{1 - \gamma / 4}} \coth (\sqrt{- \Lambda (1 - \gamma / 4)} t), \\
\phi' / \phi - \varphi' / \varphi & = F_0 \sinh^{- \tfrac{1 - \gamma /
2}{1 - \gamma / 4}} (\sqrt{- \Lambda (1 - \gamma / 4)} t),
\end{align}
  with $F_0 \in \mathbb{R}$, $t \in [t_-, t_+]$ satisfying
  \[ F_0^2 \sinh^{2 \gamma / (4 - \gamma)} (\sqrt{- \Lambda (1 - \gamma / 4)}
     t) + \Lambda (1 - \gamma / 2) \leqslant 0. \]
  If $\Lambda > 0$,
\begin{align}
\phi' / \phi + \varphi' / \varphi & = (1 - \gamma / 2)
\sqrt{\tfrac{\Lambda}{1 - \gamma / 4}} \cot (\sqrt{\Lambda (1 - \gamma /
4)} t), \\
\phi' / \phi - \varphi' / \varphi & = F_0 \sin^{- \tfrac{1 - \gamma / 2}{1
- \gamma / 4}} (\sqrt{\Lambda (1 - \gamma / 4)} t),
\end{align}
  with $F_0 \in \mathbb{R}$, $t \in [t_-, t_+]$ satisfying
  \[ F_0^2 \sin^{2 \gamma / (4 - \gamma)} (\sqrt{- \Lambda (1 - \gamma / 4)}
     t) - \Lambda (1 - \gamma / 2) \leqslant 0. \]
  The conditions on $F_0$ and $t \in [t_-, t_+]$ are to ensure that the Ricci
  curvature normal to the $\partial_t$ direction is greater than or equal to
  $\operatorname{Ric} (\partial_t, \partial_t)$.
\end{theorem}

\begin{remark}
  It is interesting to note that the rigid band for Theorem \ref{thm bw ric2}
  is a doubly warped product when $0 \leqslant \gamma < 2$; when $\gamma = 2$,
  the rigid band is a warped product; but when $2 < \gamma < 4$, there is no
  rigid band which realizes the width.
\end{remark}

Most of the proof of Theorem \ref{thm bw ric2} was laid out in {\cite[Theorems
1.1-1.2]{chai-band-2025}}. We only need an additional argument to deal with the
boundary which is very similar to that of Theorem \ref{bm bw}. We omit the
proof of Theorem \ref{thm bw ric2}, see also Remark \ref{rk omit proof ric2}.

\subsection{Applications and extensions}

We have found some applications of the band width estimate to the splitting
theorems and Geroch conjecture for manifolds with positive spectral scalar
curvature with arbitrary ends.

\begin{theorem}
  \label{thm global splitting scalar}Let $(M^n, g) = (T^{n - 1} \times
  \mathbb{R}, g)$ and $u$ a positive function on $M$, if
  \[ - \gamma \Delta u + \tfrac{1}{2} R_g u \geqslant0 \]
  for some $0 < \gamma < \tfrac{2 n}{n - 1}$ and $3 \leqslant n \leqslant7$, then
  $(M^n, g)$ is isometric to $(T^{n - 1} \times \mathbb{R}, g_{\mathbb{T}^{n -
  1}} + \mathrm{d} t^2)$ where $g_{\mathbb{T}^{n - 1}}$ is a flat metric on
  $T^{n-1}$ and $\mathrm{d} t^2$ is the metric on $\mathbb{R}$.
\end{theorem}

Below is a generalization of {\cite[Theorem 3]{chodosh-generalized-2024}} to
the spectral setting.

\begin{theorem}
  \label{thm spectral geroch arbitrary end}For any $n$-manifold \( X \) ($3 \leqslant
  n \leqslant 7$), the connected sum $M = T^n \#X$ does not admit a complete
  metric of spectral positive scalar curvature with $0 \leqslant \gamma <
  \frac{2 n}{n - 1}$.
\end{theorem}

He-Yu-Shi {\cite[Corollary 2.7]{he-positive-2025}} proved the manifold
$N\#X$ does not admit a complete metric of spectral positive scalar curvature
with $0 \leqslant\gamma < 2$, where $N$ is an enlargeable manifold.
Theorem \ref{thm spectral geroch arbitrary end} suggests that closely related results such as
\cite{chen-positive-2024} can be generalized to the settings of spectral scalar curvature.

\begin{remark}
  For $X$ closed, this theorem could be easily proved via a conformal change
  $u^{\tfrac{n - 2}{2 (n - 1)} \gamma} g$, and we only have to require that
  $\gamma < \tfrac{2 (n - 1)}{n - 2}$. We can conclude that $g$ is flat and
  $u$ is a positive constant. 
\end{remark}




In dimension 3, we have the following
which is an extension of {\cite{chodosh-splitting-2019}} to the weighted
setting.

\begin{theorem}
  \label{thm cylinder splitting}Let \( 0 \leqslant \gamma <3 \) and $(M^3, g)$ be a connected, orientable,
  complete Riemannian manifold with $- \gamma u^{- 1} \Delta_g u +
  \tfrac{1}{2} R_g \geqslant 0$ for some positive function $u$. Assume that
  $(M, g)$ contains a properly embedded surface $\Sigma \subset M$ that is
  both homeomorphic to the cylinder and absolutely weighted area-minimizing.
  Then $(M, g)$ is flat and $u$ is a constant.
\end{theorem}

Theorem \ref{thm cylinder splitting} can be easily extended to the case with
$- \gamma u^{- 1} \Delta_g u + 2\ensuremath{\operatorname{Ric}}_g$, $0
\leqslant \gamma < \tfrac{7}{2}$. We left it to the reader.

\begin{theorem}\label{thm cylinder splitting ric2}
  Theorem \ref{thm cylinder splitting} holds if the non-negativity of the spectral scalar
  curvature is replaced by the condition $- \gamma u^{- 1} \Delta_g u + 2\ensuremath{\operatorname{Ric}}_g \geqslant 0$, $0
\leqslant \gamma < \tfrac{7}{2}$. 
\end{theorem}

Interestingly, we find an application of Theorem \ref{thm cylinder splitting ric2} to the spectral version of the Milnor conjecture and establish
the following.


\begin{theorem}
  \label{thm-spectral-classify-results}Let $(M^3,g)$ be a complete oriented,
  non-compact 3-dimensional manifold with $- \gamma \Delta_M u + \operatorname{Ric}_M
  u \geqslant 0$ and $- \gamma \Delta_M u + 2\operatorname{Ric}_M
  u \geqslant 0$  for some $0 \leqslant \gamma < 2$, $u >0$. Then
  either $M$ is diffeomorphic to $\mathbb{R}^3$ or the universal cover
  $\tilde{M}^3$ of $M^3$ is isometric to the product $\bar{M}^2 \times
  \mathbb{R}$ where $\bar{M}^2$ is a complete 2-manifold with nonnegative
  Ricci curvature.
\end{theorem}
\begin{remark}
    The condition $- \gamma \Delta_M u + 2\operatorname{Ric}_M
  u \geqslant 0$ in Theorem \ref{thm-spectral-classify-results} can also be replaced by $- \gamma \Delta_M u +\frac{1}{2}R_M
  u \geqslant 0$. 
\end{remark}

{\foldedcomment{+VXTNk6l22sqWpyu}{+VXTNk6l22sqWpyv}{comment}{bk21}{1760923682}{}{From
the example in {\cite[Remark 4.1]{antonelli-sharp-2024}}, we see that the
result is not true. If the manifold has two ends, it is true. In general, we
can take the quotient of the $\mathbb{Z}$ action for a line outside a large
compact set in the Antonelli-Pozzetta-Xu example to get a manifold $\hat{N}$.
$\hat{N}$ has one end and $\pi (\hat{N}) =\mathbb{Z}$. Then the universal
cover of $\hat{N}$ is not splitting. }}

Finally, we would like to remark that there is some
freedom to consider the spectral curvature condition with an extra gradient
term $c | \nabla_g u|^2 / u^2$, $c \in \mathbb{R}$ in most of the results obtained in this article, for example,
\begin{equation} \mathcal{R}_c = - \gamma u^{- 1} \Delta_g u + \tfrac{1}{2} R_g + c \gamma
  u^{- 2} | \nabla_g u|^2 . \label{spectral with c}
  \end{equation}
One can also consider $- \gamma u^{- 1} \Delta_g u
+\text{(\textit{resp. 2})}\ensuremath{\operatorname{Ric}}_{g}+ c \gamma u^{- 2} | \nabla u|^2$.
The reason
is the observation that
\[ u^{- 2} | \nabla_g u|^2 = | \nabla_g w|^2 = | \nabla_{\Sigma} w|^2 +
   w_{\nu}^2, \]
and we can run the same procedures as done in Subsection \ref{sec rewrite}. Another approach is to observe the following
\[ - \gamma u^{- 1} \Delta_g u + c \gamma \tfrac{| \nabla u|^2}{u^2} = -
   \tfrac{\gamma}{1 - c} u^{- (1 - c)} \Delta_g u^{1 - c}, \]
which turns \eqref{spectral with c} into a spectral scalar curvature.
As
a consequence, we can generalize the results which are for $c = 0$ to
the case $c \neq 0$ with suitable range of $c$ and $\gamma$.
In particular, setting $c =
1 - \gamma / 2$, $f = - \gamma \ln u$ gives the Perelman scalar curvature
\begin{align}
P & = R_g + 2 \Delta f - | \nabla f|^2 \\
& = R_g - 2 \gamma u^{- 1} \Delta_g u + (2 \gamma - \gamma^2) u^{- 2} |
\nabla u|^2
\end{align}
which was introduced by Perelman in his gradient flow formulation of the Ricci
flow. Recently, the Perelman scalar curvature has sparkled some interest, see
{\cite{chu-non-spin-2023}} and the references therein.

The article is organized as follows:

In Section \ref{sec warp}, we introduce basics of warped $\mu$-bubbles
including the first, second variation formulas.

In Section \ref{sec classification}, we make use of the weighted minimal
hypersurfaces, in particular, we prove Theorems \ref{thm weighted fcs} and
\ref{thm fcs ric2}.

In Section \ref{sec cmc}, we develop the rigidity results Theorems \ref{thm
acg}, \ref{thm acg ric1} and \ref{thm acg ric2} by using hypersurfaces of
constant weighted mean curvature.

In Section \ref{sec bw}, by selecting suitable $\mu$ we prove several band
width estimates which include a band width interpretation of the Bonnet-Myers
theorem (Theorem \ref{bm bw}).

In the final Section \ref{sec noncompact}, we show some applications of the
band width estimates and extend some of the results in the earlier sections to
the non-compact setting.

In Appendix \ref{warped product curvature}, we record some curvature
computation for a warped product and a doubly warped product.

\

{\textbf{Acknowledgment.}} X.C. has been partially supported by the National
Research Foundation of Korea (NRF) grant funded by the Korea government (MSIT)
(No. RS-2024-00337418).

\section{Basics of warped $\mu$-bubbles}\label{sec warp}

In this section, we introduce our main technical tool, warped $\mu$-bubbles
including the first, second variation formulas and the rewrites which relate
the second variation to the spectral curvature condition. The warped
$\mu$-bubbles includes weighted minimal hypersurfaces and hypersurfaces of
constant weighted mean curvature as special cases.

\subsection{Warped $\mu$-bubble in bands}\label{sec warped bubble}

Let $\Omega$ be a Caccioppoli set which contains a neighborhood of $\partial_-
M$ and disjoint from $\partial_+ M$, a positive function $u$ and a Lipschitz
function $h \in C^{0, 1} (\bar{M})$, we define
\begin{equation}
  E (\Omega) = \int_{\partial \Omega \cap \mathrm{int} M} u^{\gamma}
  \mathrm{d} \mathcal{H}^{n - 1} - \int_{\Omega} hu^{\gamma} \mathrm{d}
  \mathcal{H}^n . \label{eq E}
\end{equation}
Let $\Omega$ be a Caccioppoli set and $\Sigma$ be a connected component of
$\partial \Omega \cap \mathrm{int} M$. Let $\Sigma_t$ be a variation of
$\Sigma$ with the variation vector field given by $Y$, $\nu$ be the normal vector of $\Sigma$ pointing to $\partial_{+}M$, and $\Omega_t$ be the
Caccioppoli set enclosed by $\partial \Omega \setminus \Sigma$ and $\Sigma_t$,
then
\begin{equation}
  \tfrac{\mathrm{d}}{\mathrm{d} t} E (\Omega_t) |_{t = 0} = \int_{\Sigma} (H +
  \gamma u^{- 1} u_{\nu} - h) \langle Y, \nu \rangle u^{\gamma} \mathrm{d}
  \mathcal{H}^{n - 1} . \label{eq first var}
\end{equation}
We say that $\Omega$ is a \text{{\itshape{warped $h$-bubble}}} if $\Omega$ is
a critical point of $E$, in particular \eqref{eq first var} vanishes for all
$Y$. If $\Omega$ is a minimizer of $E$, we call $\Omega$ a minimizing
$h$-bubble. We say $\Sigma$ a \text{{\itshape{warped $h$-hypersurface}}} if $H +
\gamma u^{- 1} u_{\nu} = h$ along $\Sigma$. If $h = 0$ along $\Sigma$, we say
that $\Sigma$ is a \text{{\itshape{warped or weighted minimal hypersurface}}}.
The terminology of warped $h$-hypersurface and warped minimal hypersurface is just
for convenience.

\begin{lemma}[Existence of warped $\mu$-bubble]\label{lem-warped-bubble}
For a Riemannian band $(M^{n},g)$ with $3\leq n\leq 7$, if either $h\to \pm\infty$ on $\partial_{\mp}M$ in the functional $E(\Omega)$ 
or 
\[h|_{\partial_{-}M}>H_{\partial_{-}M}+\gamma u^{-1}u_{\nu_{-}},\quad h|_{\partial_{+}M}<H_{\partial_{+}M}+\gamma u^{-1}u_{\nu_{+}}.\]
Then there exists an $\Omega\in \mathcal{C}$ with smooth boundary such that
\[E(\Omega)=\inf_{\Omega'\in\mathcal{C}}E(\Omega'),\]
 where $\mathcal{C}$ is defined as
\[\mathcal{C}=\{\Omega:\text{ all Caccioppoli sets } \Omega \subset M \text{ and }\Omega\triangle {\Omega}_{0}\Subset \overset{\circ}{M}\}.\]
\end{lemma}
\begin{proof}
    From \cite[Proposition 12]{chodosh-generalized-2024}, the case $h\to \pm\infty$ is proved. We only need to prove the other case: \[h|_{\partial_{-}M}>H_{\partial_{-}M}+\gamma u^{-1}u_{\nu_{-}},\quad h|_{\partial_{+}M}<H_{\partial_{+}M}+\gamma u^{-1}u_{\nu_{+}}.\]
    Since $u$ is a positive function, we have
    \begin{align*}
        u^{-\gamma/(n-1)}h|_{\partial_{-}M}&>u^{-\gamma/(n-1)}\left(H_{\partial_{-}M}+\gamma u^{-1}u_{\nu_{-}}\right)\\
        u^{-\gamma/(n-1)}h|_{\partial_{+}M}&<u^{-\gamma/(n-1)}\left(H_{\partial_{+}M}+\gamma u^{-1}u_{\nu_{+}}\right).
    \end{align*}
    We rewrite $E(\Omega)$ as
    \begin{equation*}
  E (\Omega) = \int_{\partial \Omega \cap \mathrm{int} M} u^{\gamma}
  \mathrm{d} \mathcal{H}^{n - 1} - \int_{\Omega} u^{-\gamma/(n-1)}h u^{\tfrac{n}{n-1}\gamma} \mathrm{d}
  \mathcal{H}^n . \label{eq rewrite E}
\end{equation*}
Note that $u^{\tfrac{n}{n-1}\gamma} \mathrm{d}
  \mathcal{H}^n$ is the volume element  and $u^{\gamma}
  \mathrm{d} \mathcal{H}^{n - 1}$ is the area element of the metric $\tilde{g}=u^{\tfrac{2\gamma}{n-1}}g$ respectively. And by directly computing, we obtain that 
  \[u^{-\gamma/(n-1)}\left(H_{\partial_{\pm}M}+\gamma u^{-1}u_{\nu_{\pm}}\right)\] is the mean curvature of the metric $\tilde{g}=u^{\tfrac{2\gamma}{n-1}}g$ on $\partial M_{\pm}$. Then the result follows by  \cite[Proposition 2.1]{zhu2021} and \cite[Section 5.1]{Gromov2023} for the Riemannian band $(M,\tilde{g})$ and the function $u^{-\gamma/(n-1)}h$.
\end{proof}
Given a warped $h$-hypersurface $\Sigma$, we can calculate the first variation (or
linearisation) of $H + \gamma u^{- 1} u_{\nu} - h$ along $Y$ with $\phi=\langle Y,\nu\rangle$, and we obtain
\begin{align}
& \delta (H + \gamma u^{- 1} u_{\nu} - h) \\
= & - \Delta_{\Sigma} \phi + (- |A|^2 - \operatorname{Ric} (\nu, \nu) - \gamma
u^{- 2} u_{\nu}^2 + \gamma u^{- 1} \nabla^2_{\nu \nu} u - h_{\nu}) \phi -
\gamma u^{- 1} \langle \nabla_{\Sigma} \phi, \nabla_{\Sigma} u \rangle
\\
= & - \Delta_{\Sigma} \phi + (- |A|^2 - \operatorname{Ric} (\nu, \nu) - \gamma
u^{- 2} u_{\nu}^2 + \gamma u^{- 1}  (\Delta_g u - H u_{\nu} -
\Delta_{\Sigma} u) - h_{\nu}) \phi \\
& \quad - \gamma u^{- 1} \langle \nabla_{\Sigma} \phi, \nabla_{\Sigma} u
\rangle := L_{\Sigma} \phi . \label{eq first variation H tilde}
\end{align}
In the above, we have used that $\nabla^2_{\nu \nu} = \Delta_g u - H u_{\nu} -
\Delta_{\Sigma} u$. It follows immediately that given a critical point of $E$,
we have the second variation
\[ \tfrac{\mathrm{d}^2}{\mathrm{d} t^2} E (\Omega_t) |_{t = 0} = \int_{\Sigma}
   u^{\gamma} \phi L_{\Sigma} \phi, \label{eq second var} \]
where $\Sigma$, $\Sigma_t$, $\Omega_t$ and $L_{\Sigma}$ are given as above,
see {\cite{AX2024,chai-band-2025}}.

\subsection{Variation of a warped $h$-hypersurface}

\begin{definition}
  \label{def stable warp bdry}We say that $\Sigma$ a warped $h$-hypersurface is
  stable if there exists a positive function $\phi$ such that $L \phi \geqslant
  0$.
\end{definition}

The eigenvalue with least real part is called the \text{{\itshape{principal}}}
\text{{\itshape{eigenvalue}}}. By Krein-Rutman theorem (see
{\cite{andersson-local-2005}}), the principal eigenvalue of $L_{\Sigma}$ is
real and the corresponding eigenfunction has a sign which we choose here to be
positive. Equivalently, if $\Sigma$ is stable, then the principal eigenvalue
is non-negative.

\begin{lemma}
  \label{lm non stable to barrier}If $\Sigma$ is a non-stable warped
  $h$-hypersurface, then there exists a hypersurface $\Sigma_-$ which lies in the
  side of $\Sigma$ which $\nu$ points into and $H_{\Sigma_-} + \gamma u^{- 1}
  \langle \nabla u, \nu_{\Sigma_-} \rangle - h < 0$.
\end{lemma}

\begin{proof}
Let $\phi > 0$ be the principal eigenfunction of
$L_{\Sigma}$. Since $\Sigma$ is non-stable, $L \phi < 0$. Let $Y$ be a vector
field defined in an open neighborhood of $\Sigma$ and such that $Y = \phi \nu$
along $\Sigma$ and $\Phi_t$ be the local flow of $Y$, set $\Sigma_t = \Phi_t
(\Sigma)$. By the Taylor expansion,
\[ H_{\Sigma_t} + \gamma u^{- 1} \langle \nabla u, \nu_{\Sigma_t} \rangle - h
   = tL_{\Sigma} \phi + O (t^2) < 0 \]
for all $t > 0$ sufficiently small. Taking $\Sigma_- = \Sigma_t$ for a fixed
small $t$ finishes the proof.\end{proof}

\begin{lemma}
  \label{lm foliation}If $\Sigma$ is a warped $h$-hypersurface such that $\delta (H
  + \gamma u^{- 1} u_{\nu} - h) = - \Delta_{\Sigma} \phi$, then there exists a
  foliation $\{\Sigma_t \}_{t \in (- \varepsilon, \varepsilon)}$ such that
  $H_{\Sigma_t} + \gamma u^{- 1} u_{\nu_t} - h$ is constant along each
  $\Sigma_t$ (i.e., depending only on $t$).
\end{lemma}

For the proof, see {\cite[Lemma 3.4]{chai-band-2025}} where actually only the
facts that $H + \gamma u^{- 1} u_{\nu} - h = 0$ and that $\delta (H + \gamma
u^{- 1} u_{\nu} - h) = - \Delta_{\Sigma} \phi$ were needed. In particular, such a
foliation exists if $\Sigma$ is infinitesimally rigid. Infinitesimal
rigidity of a warped $h$-hypersurface $\Sigma$ is a condition stronger than the
condition that $\Sigma$ satisfies that $\delta (H + \gamma u^{- 1} u_{\nu} - h) = -
\Delta_{\Sigma} \phi$. However, the infinitesimal rigidity is a condition which
depends on the context, and to save the bother of stating the condition of
infinitesimal rigidity for every case, we will refer to Lemma \ref{lm
foliation}.

We find it useful to have the following (cf. {\cite[Lemma
5.2]{andersson-area-2009}}).

\begin{lemma}
  \label{lm flow to barrier}If $\Sigma$ satisfies $H + \gamma u^{- 1} u_{\nu}
  - h \lneqq 0$, there exists a hypersurface $\Sigma_-$ near $\Sigma$ lying to
  the side which $\nu$ points into such that $H_{\Sigma_-} + \gamma u^{- 1}
  u_{\nu_-} - h < 0$.
\end{lemma}

\begin{proof}
We run the mean curvature flow
\[ \partial_t x = - (H + \gamma u^{- 1} u_{\nu} - h) \nu, \text{ } x \in S \]
starting from $\Sigma$. Here $S$ is a manifold diffeomorphic to $\Sigma$. Let
$\Sigma_t = x (t, S)$. By writing the equation as a graph of a function $u$
over $\Sigma$, we see only $- H \nu$ contains second order derivatives of $u$,
hence the flow is a quasi-linear parabolic equation. By standard theory, the
flow exists in a short time interval $[0, t_0)$. We have the evolution
equation for $\tilde{H}$ that
\begin{align}
& \partial_t  \tilde{H} \\
= & \Delta_{\Sigma_t}  \tilde{H} - (- |A|^2 - \operatorname{Ric} (\nu, \nu) -
\gamma u^{- 2} u_{\nu}^2 + \gamma u^{- 1} \nabla^2_{\nu \nu} u - h_{\nu})
\tilde{H} + \gamma u^{- 1} \langle \nabla_{\Sigma_t}  \tilde{H},
\nabla_{\Sigma_t} u \rangle \\
= : & \Delta_{\Sigma_t}  \tilde{H} - Q (x, t) \tilde{H} + \gamma u^{- 1}
\langle \nabla_{\Sigma_t}  \tilde{H}, \nabla_{\Sigma_t} u \rangle
\end{align}
using the short hand $\tilde{H} = H + \gamma u^{- 1} u_{\nu} - h$ and the
first variation of $\tilde{H}$ \eqref{eq first variation H tilde}. Here, $x
\in \Sigma_t$. By the parabolic regularity theory, \(Q\) is smooth on \([0,t_0)\), and we can assume that $Q$ is bounded
on $[0, t_0 / 2]$. We take
\[ K > \max_{x \in \Sigma_t, t \in [0, t_0]} |Q (x, t) |. \]
Then
\[ (\partial_t - \Delta_{\Sigma_t}) (e^{- Kt}  \tilde{H}) = \gamma u^{- 1}
   \langle \nabla_{\Sigma_t} (e^{- Kt}  \tilde{H}), \nabla_{\Sigma_t} u
   \rangle - (Q + K) (e^{- Kt}  \tilde{H}), \]
and the coefficient of the zeroth term is negative. Hence, by the strong
maximum principle of parabolic equations, $e^{- Kt}  \tilde{H} < 0$ for all $t
\in (0, t_0 / 2]$. Take any $\Sigma_t$, $t \in [0, t_0 / 2)$ as $\Sigma_-$
would suffice.\end{proof}

\subsection{Rewrite of second variation}\label{sec rewrite}

\begin{lemma}
  \label{lm rewrite prelim}The second variation \eqref{eq second var} of the
  functional $E (\Omega)$ can be rewritten as
\begin{align}
\tfrac{\mathrm{d}^2}{\mathrm{d} t^2} E (\Omega_t) |_{t = 0} & =
\tfrac{4}{4 - \gamma} \int_{\Sigma} | \nabla_{\Sigma} \psi |^2 -
\int_{\Sigma} (1 - \tfrac{\gamma}{4}) \gamma \left| \psi \nabla_{\Sigma} w
- \tfrac{1}{2 (1 - \gamma / 4)} \nabla_{\Sigma} \psi \right|^2 \\
& \quad + \int_{\Sigma} (\gamma u^{- 1} \Delta_g u - (|A|^2 +
\operatorname{Ric} (\nu, \nu))) \psi^2 \\
& \quad - \int_{\Sigma} (\gamma Hw_{\nu} + h_{\nu} + \gamma w_{\nu}^2)
\psi^2,
\end{align}
  where $\psi = \phi u^{\gamma / 2}$ and $w = \log u$.
\end{lemma}

\begin{proof}
From {\cite[Lemma 2.4]{chai-band-2025}}, we see
\begin{align}
\tfrac{\mathrm{d}^2}{\mathrm{d} t^2} E (\Omega_t) |_{t = 0} & =
\int_{\Sigma} | \nabla_{\Sigma} \psi |^2 + \int_{\Sigma} (\gamma \psi
\langle \nabla_{\Sigma} w, \nabla_{\Sigma} \psi \rangle +
(\tfrac{\gamma^2}{4} - \gamma) \psi^2 | \nabla_{\Sigma} w|^2) \label{eq
cs25} \\
& \quad + \int_{\Sigma} (\gamma u^{- 1} \Delta_g u - (|A|^2 + \operatorname{Ric}
(\nu, \nu))) \psi^2 \\
& \quad - \int_{\Sigma} (\gamma Hw_{\nu} + h_{\nu} + \gamma w_{\nu}^2)
\psi^2 .
\end{align}
The following identity
\begin{align}
& (\tfrac{\gamma^2}{4} - \gamma) \psi^2 | \nabla_{\Sigma} w|^2 + \gamma
\psi \langle \nabla_{\Sigma} w, \nabla_{\Sigma} \psi \rangle \\
= & \tfrac{1}{4} (1 - \tfrac{\gamma}{4})^{- 1} \gamma | \nabla_{\Sigma} \psi
|^2 - (1 - \tfrac{\gamma}{4}) \gamma \left| \psi \nabla_{\Sigma} w +
\tfrac{1}{2 (\gamma / 4 - 1)} \nabla_{\Sigma} \psi \right|^2
\end{align}
finishes the proof.\end{proof}

Now we set $w = \log u$ and
\begin{equation}
  Z = - \gamma u^{- 1} \Delta_g u + (|A|^2 + \operatorname{Ric} (\nu, \nu)) + \gamma
  Hw_{\nu} + h_{\nu} + \gamma w_{\nu}^2 . \label{eq Z}
\end{equation}
First, in proving various band width estimates, the function $h$ will be
chosen as $\eta \circ \rho$ where $\eta$ is a non-increasing function and
$\rho$ a Lipschitz function with $| \nabla \rho | \leqslant 1$; we can choose
$\eta$ constant in other cases. So
\[ h_{\nu} = \eta' \circ \rho \langle \nabla \rho, \nu \rangle = \eta' \circ
   \rho (\langle \nabla \rho, \nu \rangle - 1) + \eta' \circ \rho . \]
For different spectral curvature conditions, we rewrite $Z$ in different ways
where $H + \gamma w_{\nu} = h$.

\text{{\bfseries{I}}}. Case $- \gamma u^{- 1} \Delta_g u + \operatorname{Ric}$
(e.g., for Theorem \ref{bm bw}) with $0 \leqslant \gamma < \tfrac{4}{n - 1}$:
we use $|A|^2 = (|A|^2 - \tfrac{1}{n - 1} H^2) + \tfrac{1}{n - 1} H^2$, $H +
\gamma w_{\nu} = h$ in \eqref{eq Z}, and with suitable rearrangement, we
obtain that
\begin{align}
W : & = (- \gamma u^{- 1} \Delta_g u + \operatorname{Ric}_{g} (\nu,\nu)) + (|A|^2 -
\tfrac{1}{n - 1} H^2) \label{eq W ric1} \\
& \quad + \tfrac{1}{n - 1} (h - \gamma w_{\nu})^2 + \gamma (h - \gamma
w_{\nu}) w_{\nu} + h_{\nu} + \gamma w_{\nu}^2 \\
& = (- \gamma u^{- 1} \Delta_g u + \operatorname{Ric}_{g} (\nu,\nu)) + (|A|^2 -
\tfrac{1}{n - 1} H^2) \\
& \quad + \tfrac{1}{n - 1} h^2 + \tfrac{n - 3}{n - 1} \gamma hw_{\nu} +
\gamma (1 - \tfrac{n - 2}{n - 1} \gamma) w_{\nu}^2 + h_{\nu} \\
& = \left( \tfrac{4 - (n - 1) \gamma}{4 (n - 1 - (n - 2) \gamma)} (\eta
\circ \rho)^2 + \eta' \circ \rho + (- \gamma u^{- 1} \Delta_g u +
\operatorname{Ric}_{g}) \right) + (\operatorname{Ric}_{g} (\nu, \nu) - \operatorname{Ric}_{g}) \\
& \quad + (|A|^2 - \tfrac{1}{n - 1} H^2) \\
& \quad + (1 - \tfrac{n - 2}{n - 1} \gamma) \gamma (w_{\nu} + \tfrac{n -
3}{2 (n - 1 - (n - 2) \gamma)} h)^2 \\
& \quad + \eta' \circ \rho (\langle \nabla \rho, \nu \rangle - 1),
\label{eq rewrite ric1}
\end{align}
and $Z$, $W$ are related by
\[ Z = W + \tfrac{1}{n - 1} \tilde{H}^2 + \tilde{H} (\tfrac{2}{n - 1} (h -
   \gamma w_{\nu}) + \gamma w_{\nu}) . \]

\text{{\bfseries{II}}}. Case $- \gamma u^{- 1} \Delta_g u + \tfrac{1}{2} R_g$
(e.g., for Theorem \ref{thm bw spec scalar}): Using the Schoen-Yau's rewrite
{\cite{schoen-proof-1979}}
\begin{align}
& 2 (|A|^2 + \operatorname{Ric}_{g} (\nu,\nu)) = H^2 + |A|^2 + R_g - R_{\Sigma} \label{eq
sy rewrite} \\
= & (|A|^2 - \tfrac{1}{n - 1} H^2) + \tfrac{n}{n - 1} H^2 + R_g - R_{\Sigma}
\end{align}
of $\operatorname{Ric}_{g} (\nu,\nu)$ in \eqref{eq Z}. And with a similar argument as in the
previous case, we obtain that
\begin{align}
W := & - \tfrac{1}{2} R_{\Sigma} + \left( \tfrac{2 n - (n - 1) \gamma}{4 (n
- 1) - 2 (n - 2) \gamma} (\eta \circ \rho)^2 + \eta' \circ \rho + (- \gamma
u^{- 1} \Delta_g u + \tfrac{1}{2} R_g) \right) + \tfrac{1}{2} (|A|^2 -
\tfrac{1}{n - 1} H^2) \\
& \quad + (1 - \tfrac{n - 2}{2 (n - 1)} \gamma) (w_{\nu} - \tfrac{1}{(2 (n
- 1) - (n - 2) \gamma)} h)^2 \\
& \quad + \eta' \circ \rho (\langle \nabla \rho, \nu \rangle - 1) .
\end{align}
And $Z$ and $W$ are related by
\begin{equation}
  Z = W + \tfrac{n}{2 (n - 1)} \tilde{H}^2 + \tfrac{1}{n - 1} \tilde{H} (n h -
  \gamma w_{\nu}) . \label{eq Z W in sc}
\end{equation}

\text{{\bfseries{III}}}. Case $- \gamma u^{- 1} \Delta_g u + (n - 1)
\operatorname{Ric}_{g}$ (e.g., for Theorem \ref{thm bw ric2}): let $\{e_i \}_{1
\leqslant i \leqslant n - 1}$ be an orthonormal frame along $\Sigma$, and
using
\begin{equation}
  \operatorname{Ric}_{g} (\nu, \nu) + |A|^2 = \sum_{1\leqslant i \leqslant n - 1}
  \ensuremath{\operatorname{Ric}}_{g} (e_i, e_i) + H^2 - R_{\Sigma}, \label{eq nD
  zhu}
\end{equation}
(as easily seen by the definition of the scalar curvature, the above is equivalent to \eqref{eq sy rewrite}; see also
{\cite[(5.2)]{zhu2021}}) and following arguments in the previous cases, we
obtain
\begin{align}
W = & - R_{\Sigma} + ((1 - \tfrac{1}{4} \gamma) (\eta \circ \rho)^2 + \eta'
\circ \rho + (- \gamma u^{- 1} \Delta_g u + (n - 1) \operatorname{Ric}_{g})) \label{eq
W ric2} \\
& \quad + (\sum_{i \leqslant n - 1} \ensuremath{\operatorname{Ric}}_{g} (e_i) -
(n - 1) \ensuremath{\operatorname{Ric}}_{g}) + \gamma (w_{\nu} - \tfrac{1}{2}
h)^2 \\
& \quad + \eta' \circ \rho (\langle \nabla \rho, \nu \rangle - 1),
\end{align}
and $Z$, $W$ are related by
\begin{equation}
  Z = W + \tilde{H}^2 + \tilde{H} (2 h - \gamma w_{\nu}) . \label{eq Z W ric
  2}
\end{equation}
\begin{remark}
  Note that $Z = W$ when $H + \gamma w_{\nu} = h$ along $\Sigma$ in every case.
\end{remark}

\section{Classification of stable weighted minimal surfaces}\label{sec
classification}

In this section, we show some immediate applications of weighted minimal
surfaces, which are the simplest case of the warped $h$-bubbles by taking $h$
to be identically zero in dimension 3. In particular, we prove Theorem
\ref{thm weighted fcs}.

\subsection{Classification of stable surfaces}

\begin{proof}[Proof of Theorem \ref{thm weighted fcs}]
  For a compact, stable weighted minimal surface $\Sigma$, the second
  variation \eqref{eq second var} is non-negative, so the rewrite (Lemma
  \ref{lm rewrite prelim} and case \text{{\bfseries{II}}} in Subsection
  \ref{sec rewrite}) yields
  \begin{equation}
    0 \leqslant \tfrac{4}{4 - \gamma} \int_{\Sigma} | \nabla_{\Sigma} \psi |^2
    - \int_{\Sigma} (1 - \tfrac{\gamma}{4}) \gamma \left| \psi \nabla_{\Sigma}
    w - \tfrac{1}{2 (1 - \gamma / 4)} \nabla_{\Sigma} \psi \right|^2 +
    \tfrac{1}{2} \int_{\Sigma} R_{\Sigma} \psi^2 - \int_{\Sigma} (Z +
    \tfrac{1}{2} R_{\Sigma}) \psi^2 . \label{eq rewritten stability}
  \end{equation}
  We note that $Z + \tfrac{1}{2} R_{\Sigma} \geqslant 0$ by the assumptions,
  by taking $\psi \equiv 1$ and using the Gauss-Bonnet theorem, $2 \pi \chi
  (\Sigma) \geqslant 0$. Hence, $\Sigma$ can only be a sphere or a torus.
  
  In the case of a torus, then $\nabla_{\Sigma} w = 0$ and $Z + \tfrac{1}{2}
  R_{\Sigma} = 0$ (which forces $- \gamma u^{- 1} \Delta_g u + \tfrac{1}{2}
  R_g = 0$, $A = \tfrac{1}{n - 1} H$ and $w_{\nu} = 0$). Let
  $\tilde{L}_{\Sigma} = - \tfrac{4}{4 - \gamma} \Delta_{\Sigma} + \tfrac{1}{2}
  R_{\Sigma}$. Since taking $\psi = 1$ in \eqref{eq rewritten stability}
  implies that the right hand side of \eqref{eq rewritten stability} must
  vanish, hence the eigenvalue of $\tilde{L}_{\Sigma}$ is zero and the
  corresponding eigenfunction is 1, that is, $0 = \tilde{L}_{\Sigma} \psi =
  \tfrac{1}{2} R_{\Sigma}$. Hence $\Sigma$ is flat. We conclude that $\Sigma$
  satisfies the assumptions of Lemma \ref{lm foliation} with \( h=0 \).
  
  Now we show that $(M, g)$ locally splits if $\Sigma$ is weighted
  area-minimizing. Using Lemma \ref{lm foliation}, there exists a foliation
  $\{\Sigma_t \}_{t \in (- \varepsilon, \varepsilon)}$ of constant $H + \gamma
  u^{- 1} u_{\nu}$ near $\Sigma$ and $\Sigma_0 = \Sigma$. By \eqref{eq first
  variation H tilde},
  \[ \phi_t^{- 1} \tilde{H}' (t) = - \phi_t^{- 1} \Delta_{\Sigma_t} \phi_t -
     \gamma \phi_t^{- 1} \langle \nabla_{\Sigma_t} w, \nabla_{\Sigma_t} \phi_t
     \rangle - \gamma u^{- 1} \Delta_{\Sigma_t} u - Z. \]
  Setting $\xi_t$ to be $\phi_t = u^{- \gamma / 2} e^{\xi_t}$ and using
  \eqref{eq Z W in sc}, we see
\begin{align}
& \phi_t^{- 1} \tilde{H}' (t) \\
= & - | \nabla_{\Sigma_t} \xi_t |^2 - \Delta_{\Sigma_t} \xi_t +
(\tfrac{\gamma^2}{4} - \gamma) | \nabla_{\Sigma_t} w|^2 -
\tfrac{\gamma}{2} \Delta_{\Sigma_t} w + \tfrac{1}{2} R_{\Sigma_t} - (W +
\tfrac{1}{2} R_{\Sigma_t}) \\
& \quad - \tfrac{3}{4} \tilde{H}^2 - \tfrac{1}{2} \tilde{H} (3 h - \gamma
w_{\nu}) .
\end{align}
  By integration on both sides, the divergence theorem and the Gauss-Bonnet
  theorem,
\begin{align}
& \tilde{H}' (t) \int_{\Sigma_t} \phi_t^{- 1} \\
\leqslant & - \int_{\Sigma_t} (| \nabla_{\Sigma_t} \xi_t |^2 +
(\tfrac{\gamma^2}{4} - \gamma) | \nabla_{\Sigma_t} w|^2) + 2 \pi \chi
(\Sigma_t) - \int_{\Sigma_t} (W + \tfrac{1}{2} R_{\Sigma_t}) -
\tfrac{1}{2} \tilde{H} \int_{\Sigma_t} (3 h - \gamma w_{\nu}) \\
\leqslant & 2 \pi \chi (\Sigma_t) - \tfrac{1}{2} \tilde{H} \int_{\Sigma_t}
(3 h - \gamma w_{\nu}) \\
= & - \tfrac{1}{2} \tilde{H} \int_{\Sigma_t} (3 h - \gamma w_{\nu})
\end{align}
  where we have used $0 \leqslant \gamma < 3$ and that $W + \tfrac{1}{2}
  R_{\Sigma_t} \geqslant 0$. Hence,
  \[ \tilde{H}' (t) \leqslant - \tfrac{1}{2} \tilde{H} \frac{\int_{\Sigma_t}
     (3 h - \gamma w_{\nu})}{\int_{\Sigma_t} \phi_t^{- 1}}, \]
  which by solving we conclude that $\tilde{H} \leqslant 0$ for all $t \in [0,
  \varepsilon)$ and $\tilde{H} \geqslant 0$ for all $t \in (- \varepsilon,
  0]$. By \eqref{eq first var}, $\Sigma_t$ is also weighted area-minimizing.
  Hence, the argument works for $\Sigma$ also works for $\Sigma_t$.
  
  Since $H + \gamma u^{- 1} u_{\nu} = 0$ and $w_{\nu} = 0$, $\Sigma_t$ is
  minimal. From $A - \tfrac{1}{n - 1} H = 0$, we see $\Sigma_t$ is totally
  geodesic. So $(M, g)$ locally splits and hence $u$ is a positive constant.
  
  In the case that $\Sigma$ is non-compact, \eqref{eq rewritten stability}
  holds for all $\psi \in C_c^{\infty} (\Sigma)$ which implies that the
  operator
  \[ L = - \tfrac{4}{4 - \gamma} \Delta_{\Sigma} + \tfrac{1}{2} R_{\Sigma} -
     (Z + \tfrac{1}{2} R_{\Sigma}) \]
  in the condition that $0 \leqslant \gamma < 3$, is non-negative by \eqref{eq
  rewrite ric1}. By {\cite[Theorems 1.1-1.3]{berard-inverse-2014}}, $\Sigma$
  is conformally equivalent to either the complex plane $\mathbb{C}$ or the
  cylinder $\mathbb{A}$; in the case of a cylinder, $\Sigma$ is flat and $Z +
  \tfrac{1}{2} R_{\Sigma} = 0$ which implies $- \gamma u^{- 1} \Delta_g u +
  \tfrac{1}{2} R_g = 0$, $A = \tfrac{1}{n - 1} H$ and $w_{\nu} = 0$. By taking
  a simple cutoff $\psi$ approximating 1 in \eqref{eq rewritten stability}, we
  see $\nabla_{\Sigma} w = 0$.
\end{proof}

With almost the same proof, we can give a proof of Theorem \ref{thm fcs ric2}.

\begin{proof}[Proof of Theorem \ref{thm fcs ric2}]
  It suffices to replace the rewrites of $Z$ and $W$ by case
  \text{{\bfseries{III}}} of Subsection \ref{sec rewrite}. The proof is almost
  verbatim except deriving the consequences of $\nabla_{\Sigma} w = 0$,
  $R_{\Sigma} = 0$ and $Z + R_{\Sigma} = 0$. The condition $Z + R_{\Sigma} =
  0$ implies that
  $\ensuremath{\operatorname{Ric}}_{g}=\ensuremath{\operatorname{Ric}}_{g} (e_i,
  e_i)$, $- \gamma u^{- 1} \Delta_g u + 2\ensuremath{\operatorname{Ric}}_{g}= 0$
  and $w_{\nu} = 0$, where $e_i$ is any tangent vector of $\Sigma$. Using
  {\cite[Subsection 3.4]{chai-band-2025}}, $(M, g)$ is locally a doubly warped
  product, say $\mathrm{d} t^2 + \phi (t)^2 \mathrm{d} s_1^2 + \varphi (t)^2
  \mathrm{d} s_2^2$ which satisfies \eqref{eq F G} and \eqref{eq G F}. Denote
  $t$-level set by $\Sigma_t$. First,
  \[ H = \phi' / \phi + \varphi' / \varphi = - \gamma w_{\nu} = 0. \]
  By \eqref{eq G F}, $\phi' / \phi = \varphi' / \varphi$. Hence, $\phi' / \phi
  = \varphi' / \varphi = 0$ giving that both $\phi$ and $\varphi$ are
  constants.
\end{proof}

\section{Hypersurfaces of constant weighted mean curvature}\label{sec cmc}

In this section, we prove Theorems \ref{thm acg}, \ref{thm acg ric1} and
\ref{thm acg ric2} by choosing $h$ to be a constant in the definitions of
warped $h$-bubble.

Note that now we have the structure of a band, that is, two boundary
components. And, along the boundaries, the barrier condition is not strict. We
need to address the existence of a warped $h$-bubble first.

\subsection{Spectral scalar curvature case}

\begin{theorem}
  \label{thm non-exist strict barrier}Let $\gamma$, $\Lambda$, $\alpha$,
  $\beta$ and \( \eta \) be as in Theorem \ref{thm acg}. There does not exist a torical band
  $(M, g)$ such that $- \gamma u^{- 1} \Delta_g u + \tfrac{1}{2} R_g \geqslant
  \Lambda$, $H + \gamma u^{- 1} u_{\nu} > (n - 1) \alpha + \gamma \beta$ along
  $\partial_+ M$ and $H + \gamma u^{- 1} u_{\nu} < (n - 1) \alpha + \gamma
  \beta$.
\end{theorem}

\begin{proof}
We assume that there exists such a torical band $(M,
g)$. Since $\eta = (n - 1) \alpha + \beta \gamma$, using the strict barrier condition (see Lemma\ref{lem-warped-bubble}), we have a minimizing warped
$\eta$-bubble $\Omega$.  Let
$\Sigma$ be one of the connected components of $\partial \Omega \backslash
\partial_- M$.

Using the stability inequality with the rewrite \text{{\bfseries{II}}} of
Subsection \ref{sec rewrite}, we obtain
\begin{equation}
  0 \leqslant \int_{\Sigma} \left( \tfrac{4}{4 - \gamma} | \nabla_{\Sigma}
  \psi |^2 + \tfrac{1}{2} R_{\Sigma} \psi^2 \right) - \int_{\Sigma} (1 -
  \tfrac{\gamma}{4}) \gamma \psi^2  \left| \nabla_{\Sigma} w - \tfrac{1}{2 (1
  - \gamma / 4)} \tfrac{\nabla_{\Sigma} \psi}{\psi} \right|^2 - \int_{\Sigma}
  (Z + \tfrac{1}{2} R_{\Sigma}) \psi^2 . \label{eq acg stability}
\end{equation}
For $n = 3$, we use an argument similar to Theorem \ref{thm weighted fcs}. We
deal with $n \geqslant 4$. Let
\begin{equation}
  L = - \tfrac{4}{4 - \gamma} \Delta_{\Sigma} + \tfrac{1}{2} R_{\Sigma} - (Z +
  \tfrac{1}{2} R_{\Sigma}) . \label{eq acg L}
\end{equation}
Let $\lambda_1$ be the first eigenvalue of $L$ and $v$ be the corresponding
eigenfunction. Note that $\lambda_1 \geqslant 0$ by \eqref{eq acg stability}.

We define a constant $\kappa$ by $4 (n - 2) \kappa / (n - 3) = 8 / (4 -
\gamma)$. By the range of $\gamma$, $\kappa \in (0, 1)$. Let $\hat{g} =
(v^{\kappa})^{4 / (n - 3)} g|_{\Sigma}$ be the conformal metric. Then the
scalar curvature of $\Sigma$ with respect to $\hat{g}$ is
\begin{align}
(v^{\kappa})^{\tfrac{n + 1}{n - 3}} R_{\Sigma} (\hat{g}) & = v^{\kappa}
R_{\Sigma} - \tfrac{4 (n - 2)}{n - 3} \Delta_{\Sigma} v^{\kappa} \\
& = v^{\kappa} (R_{\Sigma} - \tfrac{4 (n - 2)}{n - 3} \alpha v^{- 1}
\Delta_{\Sigma} v - \tfrac{4 (n - 2)}{n - 3} \kappa (\kappa - 1) v^{- 2} |
\nabla_{\Sigma} v|^2) \\
& \geqslant v^{\kappa} (\tfrac{2 L v}{v} + (Z + \tfrac{1}{2} R_{\Sigma}) -
\tfrac{4 (n - 2)}{n - 3} \kappa (\kappa - 1) v^{- 2} | \nabla_{\Sigma} v|^2)
\\
& = v^{\kappa} \left( 2 \lambda_1 + (Z + \tfrac{1}{2} R_{\Sigma}) -
\tfrac{4 (n - 2)}{n - 3} \kappa (\kappa - 1) v^{- 2} | \nabla_{\Sigma} v|^2
\right) \\
& \geqslant 0.
\end{align}
Using the assumptions, it is direct to check that $Z + \tfrac{1}{2} R_{\Sigma}
\geqslant 0$ along $\Sigma$. So $R_{\Sigma} (\hat{g}) \geqslant 0$. By the
resolution of the Geroch conjecture and that $0 < \kappa < 1$, $R_{\Sigma}
(\hat{g})$ has to vanish which implies that $\lambda_1 = 0$, $\nabla_{\Sigma}
v = 0$, $Z + \tfrac{1}{2} R_{\Sigma} = 0$ (i.e., $A - \tfrac{1}{n - 1} H = 0$,
$- \gamma u^{- 1} \Delta_g u + \tfrac{1}{2} R_g = \Lambda$, $w_{\nu} =
\tfrac{1}{2 (n - 1) - (n - 2) \gamma} \eta$) and $R_{\Sigma} = 0$. Now with
$\psi = v = 1$ in \eqref{eq acg stability} whose right hand side vanishes, we
have $\nabla_{\Sigma} w = 0$. So we have shown that $\Sigma$ is infinitesimally
rigid. It follows from Theorem \ref{lm foliation} that we have a foliation
$\{\Sigma_t \}_{t \in (- \varepsilon, \varepsilon)}$ near $\Sigma$.

\text{{\bfseries{Claim}}}: $\tilde{H}_t \leqslant 0$ for $t \in (0,
\varepsilon)$ and $\tilde{H} \geqslant 0$ for $t \in (- \varepsilon, 0)$.

To show this claim, we repeat an argument in {\cite[Lemma
4.4]{chai-band-2025}}. By \eqref{eq first variation H tilde},
\[ \phi_t^{- 1} \tilde{H}' (t) = - \phi_t^{- 1} \Delta_{\Sigma_t} \phi_t -
   \gamma \phi_t^{- 1} \langle \nabla_{\Sigma_t} w, \nabla_{\Sigma_t} \phi_t
   \rangle - \gamma u^{- 1} \Delta_{\Sigma_t} u - Z. \]
Setting $\xi_t$ to be $\phi_t = u^{- \gamma / 2} e^{\xi_t}$ and using
\eqref{eq Z W in sc}, we see
\begin{align}
& \phi_t^{- 1} \tilde{H}' (t) \\
= & - | \nabla_{\Sigma_t} \xi_t |^2 - \Delta_{\Sigma_t} \xi_t +
(\tfrac{\gamma^2}{4} - \gamma) | \nabla_{\Sigma_t} w|^2 - \tfrac{\gamma}{2}
\Delta_{\Sigma_t} w + \tfrac{1}{2} R_{\Sigma_t} - (W + \tfrac{1}{2}
R_{\Sigma_t}) \\
& \quad - \tfrac{n}{2 (n - 1)} \tilde{H}^2 - \tfrac{1}{n - 1} \tilde{H} (n
h - \gamma w_{\nu}),
\end{align}
which leads to
\[ \tilde{H}' + q_t \tilde{H} \leqslant \phi_t (- | \nabla_{\Sigma_t} \xi_{t }
   |^2 - \Delta_{\Sigma_t} \xi_t + (\tfrac{\gamma^2}{4} - \gamma) |
   \nabla_{\Sigma_t} w|^2 - \tfrac{\gamma}{2} \Delta_{\Sigma_t} w +
   \tfrac{1}{2} R_{\Sigma_t}) . \]
Here, $q_t := - \tfrac{1}{n - 1} (n h - \gamma w_{\nu}) \phi_t$. For each $t$,
so for any positive function $\varphi \in C^2 (\Sigma_t)$,
\begin{equation}
  \tilde{H}' \int_{\Sigma_t} \varphi + \tilde{H} \int_{\Sigma_t} q_t \varphi
  \leqslant \int_{\Sigma_t} \varphi \phi_t (- | \nabla_{\Sigma_t} \xi_{t } |^2
  - \Delta_{\Sigma_t} \xi_t + (\tfrac{\gamma^2}{4} - \gamma) |
  \nabla_{\Sigma_t} w|^2 - \tfrac{\gamma}{2} \Delta_{\Sigma_t} w +
  \tfrac{1}{2} R_{\Sigma_t}) . \label{ode in acg}
\end{equation}
It suffices to show that there exists a positive function $\varphi$ such that
the right hand side is non-positive. Assume the contrary, and without loss of
generality, we can replace $\varphi \phi_t$ by $\varphi^2$. First,
\begin{align}
& (| \nabla_{\Sigma_t} \xi_t |^2 + \Delta_{\Sigma_t} \xi_t) \varphi^2
\\
= & | \nabla_{\Sigma_t} \xi_t |^2 \psi^2 - 2 \langle \nabla_{\Sigma_t}
\varphi, \varphi \nabla_{\Sigma_t} \xi_t \rangle
+\ensuremath{\operatorname{div}}_{\Sigma_t} (\varphi^2 \nabla_{\Sigma_t}
\xi_t) \\
\geqslant & - | \nabla_{\Sigma_t} \varphi |^2
+\ensuremath{\operatorname{div}}_{\Sigma_t} (\varphi^2 \nabla_{\Sigma_t}
\xi_t) .
\end{align}
It follows from integration by parts that
\[ - \int (| \nabla_{\Sigma_t} \xi_t |^2 + \Delta_{\Sigma_t} \xi_t) \varphi^2
   \leqslant \int_{\Sigma_t} | \nabla_{\Sigma_t} \varphi |^2 . \]
And also
\[ \tfrac{\gamma}{2} \int_{\Sigma_t} \varphi^2 \Delta_{\Sigma_t} w = -
   \int_{\Sigma_t} \gamma \varphi \langle \nabla_{\Sigma_t} w,
   \nabla_{\Sigma_t} \varphi \rangle . \]
So
\begin{align}
& \int_{\Sigma_t} (- | \nabla_{\Sigma_t} \xi_{t } |^2 - \Delta_{\Sigma_t}
\xi_t + (\tfrac{\gamma^2}{4} - \gamma) | \nabla_{\Sigma_t} w|^2 -
\tfrac{\gamma}{2} \Delta_{\Sigma_t} w + \tfrac{1}{2} R_{\Sigma_t}) \varphi^2
\\
= & \int_{\Sigma_t} (| \nabla_{\Sigma_t} \varphi |^2 + \gamma \varphi
\langle \nabla_{\Sigma_t} w, \nabla_{\Sigma_t} \varphi \rangle +
(\tfrac{\gamma^2}{4} - \gamma) | \nabla_{\Sigma_t} w|^2 \varphi^2 +
\tfrac{1}{2} R_{\Sigma_t} \varphi^2) \\
= & \int_{\Sigma_t} (\tfrac{4}{4 - \gamma} | \nabla_{\Sigma_t} \varphi |^2 -
(1 - \tfrac{\gamma}{4}) \gamma \left| \varphi \nabla_{\Sigma_t} w -
\tfrac{1}{2 (1 - \gamma / 4)} \nabla_{\Sigma_t} \varphi \right|^2 +
\tfrac{1}{2} R_{\Sigma_t} \varphi^2) \\
> & 0.
\end{align}
Let $L = - \tfrac{4}{4 - \gamma} \Delta_{\Sigma_t} + \tfrac{1}{2}
R_{\Sigma_t}$, $\lambda_1$ be the first eigenvalue of $L$ and $v > 0$ be the
first eigenfunction. By the above inequality, $L v = \lambda_1 v > 0$. Let
$\hat{g}_t = (v^{\kappa})^{4 / (n - 3)} g|_{\Sigma_t}$ be the conformal
metric. Then the scalar curvature of $\Sigma_t$ with respect to $\hat{g}$ is
\begin{align}
(v^{\kappa})^{\tfrac{n + 1}{n - 3}} R_{\Sigma_t} (\hat{g}_t) & = v^{\kappa}
R_{\Sigma_t} - \tfrac{4 (n - 2)}{n - 3} \Delta_{\Sigma_t} v^{\kappa}
\\
& = v^{\kappa} (R_{\Sigma_t} - \tfrac{4 (n - 2)}{n - 3} \kappa v^{- 1}
\Delta_{\Sigma_t} v - \tfrac{4 (n - 2)}{n - 3} \kappa (\kappa - 1) v^{- 2} |
\nabla_{\Sigma_t} v|^2) \\
& \geqslant v^{\kappa} (\tfrac{2 L v}{v} - \tfrac{4 (n - 2)}{n - 3} \kappa
(\kappa - 1) v^{- 2} | \nabla_{\Sigma_t} v|^2) \\
& = v^{\kappa} \left( 2 \lambda_1 - \tfrac{4 (n - 2)}{n - 3} \kappa (\kappa
- 1) v^{- 2} | \nabla_{\Sigma_t} v|^2 \right) \\
& > 0.
\end{align}
This is in contradiction with the resolution of the Geroch conjecture. Hence,
for each $t$, there exists some $\varphi$ such that the right hand side of
\eqref{ode in acg} is non-negative. Choosing such $\varphi$ for each $t$, and
solving \eqref{ode in acg}, we finish the proof of the claim.

By the first variation \eqref{eq first var}, every $\Sigma_t$ gives rise to a
minimizer to the warped functional. And the rigidity extends to all $M$, which
means that there exists a leaf that would meet $\partial_- M$ tangentially. However, the
strong maximum principle implies that $\partial_- M$ satisfies $H + \gamma
u^{- 1} u_{\nu} = \eta$. This is a contradiction to the assumption of the
existence of $(M, g)$.\end{proof}

Now we are ready to finish the proof of Theorem \ref{thm acg}.

\begin{proof}[Proof of Theorem \ref{thm acg}]
  First, for $\partial_- M$, we claim that either there exists a hypersurface
  $\Sigma_-$ near $\partial_- M$ such that $H_{\Sigma_-} + \gamma u^{- 1}
  \langle \nabla u, \nu_{\Sigma_-} \rangle < \eta$ or there exists a
  \text{{\itshape{maximal}}} foliation $\{\Sigma_t^- \}_{t \in [0, t_1]}$ such
  that $\Sigma_0^- = \partial_- M$ and every leaf is of vanishing $H + \gamma
  u^{- 1} u_{\nu} - \eta$. The maximality means that either the foliation
  foliates all of $M$ or if the foliation were to extend beyond
  $\Sigma_{t_1}^-$ (since $\Sigma_{t_1}^-$ is stable by Definition \ref{def
  stable warp bdry}) which will give a leaf $\Sigma_-$ with $H + \gamma u^{-
  1} u_{\nu} - \eta < 0$ in the extended foliation.
  
  Indeed, if $H + \gamma u^{- 1} u_{\nu} - \eta \lneqq 0$ along $\partial_-
  M$, we use Lemma \ref{lm flow to barrier}; If $H + \gamma u^{- 1} u_{\nu}
  - \eta = 0$ but not stable, then we use Lemma \ref{lm non stable to
  barrier}. In both cases, we obtain a hypersurface which satisfies the claim.
  
  If $\partial_- M$ is stable, then by Lemma \ref{lm foliation}, we obtain a
  foliation $\{\Sigma_t^- \}_{t \in [0, \varepsilon]}$. By the proof of
  Theorem \ref{thm non-exist strict barrier}, $H + \gamma u^{- 1} u_{\nu} -
  \eta \leqslant 0$ for every $\{\Sigma_t^- \}_{t \in [0, \varepsilon]}$.
  Either there exists some $t \in (0, \varepsilon]$ such that $H + \gamma u^{-
  1} u_{\nu} - \eta < 0$ or $H + \gamma u^{- 1} u_{\nu} - \eta = 0$ for all
  $\Sigma_t^-$, $t \in [0, \varepsilon]$. (It is worth noting that
  $\varepsilon = 0$ is also possible.) In the latter case, all leaves are
  stable by Definition \ref{def stable warp bdry}. Hence, the foliation can be
  extended beyond $\Sigma_{\varepsilon}^-$. In light of this, we can assume
  that starting from $\partial_- M$, there is a maximal foliation
  $\{\Sigma_t^- \}_{t \in [0, t_1]}$. If the union of the leaves are the
  closure of $M$, then we are done. If not, by maximality, we have a
  hypersurface $\Sigma_-$ with $H + \gamma u^{- 1} u_{\nu} - \eta < 0$
  constructed as a leaf of the foliation started from $\Sigma_{t_1}^-$ by
  Lemma \ref{lm foliation}. Hence, the claim is proved.
  
  We can argue similarly for $\partial_+ M$ to show that there exists a
  hypersurface $\Sigma_+$ near $\partial_- M$ such that $H_{\Sigma_+} + \gamma
  u^{- 1} \langle \nabla u, \nu_{\Sigma_+} \rangle > \eta$ or there exists a
  maximal foliation $\{\Sigma_t^+ \}_{t \in [0, t_2]}$ such that $\Sigma_0^+ =
  \partial_+ M$ and every leaf is of vanishing $H + \gamma u^{- 1} u_{\nu} -
  \eta$. (For this, we have to reverse the signs of the mean curvature, $\nu$
  and $\eta$.)
  
  Note that the two leaves from $\{\Sigma_t^- \}$ and $\{\Sigma_t^+ \}$ can
  only touch which by the strong maximum principle are the same leaf. It then
  implies that the two foliations are the same one and foliate all of $M$. If
  this happens, we can finish the rigidity of $M$. If not, by maximality, we
  can extend the foliations and obtain the hypersurfaces $\Sigma_-$ and
  $\Sigma_+$ which satisfies the barrier condition strictly.
  
  To summarize, either the rigidity holds for $M$ or there exists two
  hypersurfaces $\Sigma_-$ and $\Sigma_+$ which satisfies the barrier
  condition strictly. However, the latter case is ruled out by Theorem
  \ref{thm non-exist strict barrier}. Now we find the metric $g$ and $u$ based
  on the foliation. The foliation gives (as in Theorem \ref{thm non-exist
  strict barrier})
  \[ \nabla_{\Sigma_t} w = 0, \text{ } - \gamma u^{- 1} \Delta_g u +
     \tfrac{1}{2} R_g = \Lambda, \text{ } A - \tfrac{1}{n - 1} H = 0, \text{ }
     w_{\nu} = \tfrac{1}{2 (n - 1) - (n - 2) \gamma} \eta \]
  and that each $\Sigma_t$ is a flat torus.
  
  The equation $A - \tfrac{1}{n - 1} H = 0$ gives that $(M, g)$ is isometric
  to some warped product $\mathrm{d} t^2 + \phi (t)^2 g_{\mathbb{T}^{n - 1}}$,
  and we can assume that the foliation is given by the level set of the
  coordinate $t$. Note that $\nabla_{\Sigma_t} w = 0$ gives that $u$ is
  constant along each $\Sigma_t$, so $w_{\nu} = \tfrac{1}{2 (n - 1) - (n - 2)
  \gamma} \eta$ leads to ${u = e^{\eta t / (2 (n - 1) - (n - 2) \gamma)}} $
  (up to a constant). Now the equation
  \[ H + \gamma u^{- 1} u_{\nu} = \eta = (n - 1) \phi' / \phi + \gamma u^{- 1}
     u_{\nu} \]
  reduces to an ODE for the warping factor $\phi$, and (up to a constant)
  $\phi = e^{\alpha t}$.
\end{proof}

The proof of Theorem \ref{thm acg ric2} differs only by the calculation of
rigid metrics.

\begin{proof}[Proof of Theorem \ref{thm acg ric2}]
  It suffices to follow Theorem \ref{thm acg} and to replace the rewrites of
  $Z$ and $W$ by case \text{{\bfseries{III}}} of Subsection \ref{sec rewrite}.
  For dimensions $n \geqslant 4$, we need to establish a version of Theorem
  \ref{thm non-exist strict barrier}. We proceed the proof to the place where
  the range $0 \leqslant \gamma < 3 + \tfrac{1}{n - 2}$ is needed and omit the
  rest. By the existence result, we have a stable hypersurface $\Sigma$ which
  is a torus and with vanishing $H + \gamma u^{- 1} u_{\nu} - \eta$. The
  stability gives
  \begin{equation}
    0 \leqslant \int_{\Sigma} \left( \tfrac{4}{4 - \gamma} | \nabla_{\Sigma}
    \psi |^2 + R_{\Sigma} \psi^2 \right) - \int_{\Sigma} (1 -
    \tfrac{\gamma}{4}) \gamma \psi^2  \left| \nabla_{\Sigma} w - \tfrac{1}{2
    (1 - \gamma / 4)} \tfrac{\nabla_{\Sigma} \psi}{\psi} \right|^2 -
    \int_{\Sigma} (Z + R_{\Sigma}) \psi^2 . \label{eq acg stability ric2}
  \end{equation}
  Let
  \begin{equation}
    L = - \tfrac{4}{4 - \gamma} \Delta_{\Sigma} + R_{\Sigma} - (Z +
    R_{\Sigma}) . \label{eq acg L ric2}
  \end{equation}
  Let $\lambda_1$ be the first eigenvalue of $L$ and $v$ be the corresponding
  eigenfunction. Note that $\lambda_1 \geqslant 0$ by \eqref{eq acg stability
  ric2}.
  
  We define a constant $\kappa$ by $4 (n - 2) \kappa / (n - 3) = 4 / (4 -
  \gamma)$. By the range of $\gamma$, $\kappa \in (0, 1)$. Let $\hat{g} =
  (v^{\kappa})^{4 / (n - 3)} g|_{\Sigma}$ be the conformal metric. Then the
  scalar curvature of $\Sigma$ with respect to $\hat{g}$ is
\begin{align}
(v^{\kappa})^{\tfrac{n + 1}{n - 3}} R_{\Sigma} (\hat{g}) & = v^{\kappa}
R_{\Sigma} - \tfrac{4 (n - 2)}{n - 3} \Delta_{\Sigma} v^{\kappa}
\\
& = v^{\kappa} (R_{\Sigma} - \tfrac{4 (n - 2)}{n - 3} \kappa v^{- 1}
\Delta_{\Sigma} v - \tfrac{4 (n - 2)}{n - 3} \kappa (\kappa - 1) v^{- 2} |
\nabla_{\Sigma} v|^2) \\
& \geqslant v^{\kappa} (\tfrac{L v}{v} + (Z + R_{\Sigma}) - \tfrac{4 (n -
2)}{n - 3} \kappa (\kappa - 1) v^{- 2} | \nabla_{\Sigma} v|^2) \\
& = v^{\kappa} \left( \lambda_1 + (Z + R_{\Sigma}) - \tfrac{4 (n - 2)}{n
- 3} \kappa (\kappa - 1) v^{- 2} | \nabla_{\Sigma} v|^2 \right)
\\
& \geqslant 0.
\end{align}
  In order to obtain $\nabla_{\Sigma} v = 0$, $\lambda_1 = 0$ and $Z +
  R_{\Sigma} = 0$, we need $\kappa < 1$ which is equivalent to $\gamma < 3 +
  \tfrac{1}{n - 2}$. We omit the rest of the proof which is almost the same
  with Theorem \ref{thm acg} except that we obtain $H + \gamma u^{- 1} u_{\nu}
  = \eta$, $\nabla_{\Sigma} w = 0$, $R_{\Sigma} = 0$ and $Z + R_{\Sigma} = 0$
  for the foliation. The condition $Z + R_{\Sigma} = 0$ implies that
  $\ensuremath{\operatorname{Ric}}_{g}=\ensuremath{\operatorname{Ric}}_{g} (e_i,
  e_i)$, $- \gamma u^{- 1} \Delta_g u + 2\ensuremath{\operatorname{Ric}}_{g}=
  \Lambda$ and $w_{\nu} = \tfrac{1}{2} \eta$, where $e_i$ is any tangent
  vector of $\Sigma$.
  
  In dimension 3, using {\cite[Subsection 3.4]{chai-band-2025}}, $(M, g)$ is
  locally a doubly warped product, say $\mathrm{d} t^2 + \phi (t)^2 \mathrm{d}
  s_1^2 + \varphi (t)^2 \mathrm{d} s_2^2$. Denote $t$-level set by $\Sigma_t$.
  First, $w_{\nu} = \tfrac{1}{2} \eta$ gives $u = e^{\eta t / 2}$, and
  \begin{equation}
    H = \phi' / \phi + \varphi' / \varphi = \eta - \gamma w_{\nu} = (1 -
    \tfrac{\gamma}{2}) \eta . \label{eq H acg ric2}
  \end{equation}
  By the Ricci curvatures of a doubly warped product metric in Appendix
  \ref{app doubly}, $\ensuremath{\operatorname{Ric}} (\partial_t, \partial_t)
  \geqslant \ensuremath{\operatorname{Ric}} (e_i, e_i)$ gives
  \[ 0 \geqslant (\phi' / \phi + \varphi' / \varphi')' + (\phi' / \phi -
     \varphi' / \varphi)^2 = (\phi' / \phi - \varphi' / \varphi)^2 \]
  see \eqref{eq G F}. So $\phi' / \phi = \varphi' / \varphi = \tfrac{1}{2} (1
  - \tfrac{\gamma}{2}) \eta$ by \eqref{eq H acg ric2}. Then up to a factor we
  can choose $\phi = \varphi = e^{(1 - \gamma / 2) \eta / 2}$.
\end{proof}

\begin{remark}
  In dimensions $n \geqslant 4$, we cannot find a metric structure like a
  doubly warped product.
\end{remark}

\subsection{Spectral Ricci curvature case}

Before we prove Theorem \ref{thm acg ric1}, we need an analogous Theorem
\ref{thm non-exist strict barrier}.

\begin{theorem}
  \label{thm non-exist strict barrier acg ric}Let $\gamma$, $\Lambda$,
  $\alpha$ and $\beta$ be as in Theorem \ref{thm acg ric1}. There does not
  exist a band $(M, g)$ such that $- \gamma u^{- 1} \Delta_g u
  +\ensuremath{\operatorname{Ric}}_g \geqslant \Lambda$, $H + \gamma u^{- 1}
  u_{\nu} > (n - 1) \alpha + \gamma \beta$ along $\partial_+ M$ and $H +
  \gamma u^{- 1} u_{\nu} < (n - 1) \alpha + \gamma \beta$.
\end{theorem}

\begin{proof}
Let $\Sigma = \partial \Omega \cap \mathrm{int} M$,
then $\Sigma$ is a stable warped $h$-hypersurface. Then using $- \gamma u^{- 1}
\Delta_g u +\ensuremath{\operatorname{Ric}}_g \geqslant \Lambda$ and the
condition on $\eta$, the stability \eqref{eq second var} gives
\[ 0 \leqslant \tfrac{4}{4 - \gamma} \int_{\Sigma} | \nabla_{\Sigma} \psi |^2
   - \int_{\Sigma} (1 - \tfrac{\gamma}{4}) \left| \psi \nabla_{\Sigma} w -
   \tfrac{1}{2 (1 - \gamma / 4)} \nabla_{\Sigma} \psi \right|^2 -
   \int_{\Sigma} Z \psi^2, \]
where $Z$ is defined in \eqref{eq Z}. When $\Sigma$ is an $h$-hypersurface, $Z = W
\geqslant 0$. Hence taking $\psi = 1$, we find that $Z = W = 0$ and
$\nabla_{\Sigma} w = 0$, and hence by \eqref{eq first variation H tilde},
$\Sigma$ is infinitesimally rigid. By Lemma \ref{lm foliation}, we can
construct a foliation $\{\Sigma_t \}_{t \in (- \varepsilon, \varepsilon)}$ for
some $\varepsilon > 0$.

\text{{\bfseries{Claim}}}: for each $t \in (- \varepsilon, \varepsilon)$,
$\tilde{H} (t) := H + \gamma u^{- 1} u_{\nu} - h$ on $\{\Sigma_t \}$ satisfies
\[ \tilde{H} (t) \leqslant 0 \text{ for } t \in (0, \varepsilon) \text{ and }
   \tilde{H} (t) \geqslant 0 \text{ for } t \in (- \varepsilon, 0) . \]
Let $Y$ be the variational vector field of the foliation $\{\Sigma_t \}_{t \in
(- \varepsilon, \varepsilon)}$, let $\phi_t = \langle Y, \nu_t \rangle$. Using
the first variation \eqref{eq first variation H tilde} of $\tilde{H} (t)$ , we
see that
\begin{equation}
  \phi_t^{- 1}  \tilde{H}' = - \phi_t^{- 1} \Delta_{\Sigma_t} \phi_t - u^{- 1}
  \Delta_{\Sigma_t} u - \gamma \phi^{- 1}_t \langle \nabla_{\Sigma_t} w,
  \nabla_{\Sigma_t} \phi_t \rangle - Z_t \label{eq variation with Pt}
\end{equation}
where $w = \log u$ and $Z_t$ is defined as in \eqref{eq Z} for $\Sigma_t$. Let
$W$ be defined as \eqref{eq W ric1} for $\Sigma_t$. By the rewrite of $Z_t$
when $\tilde{H}$ vanishes, $Z_t = W_t$ where $W_t$ is given in \eqref{eq W
ric1}. When $\tilde{H}$ might not vanish, $Z_t$ and $W_t$ are related by
\[ Z_t = W_t + \tfrac{1}{n - 1} \tilde{H}^2 + \tilde{H} q_t \]
where $q_t := \tfrac{2 (h - \gamma w_{\nu_t})}{n - 1} + \gamma w_{\nu_t}$. We
set $\phi_t = u^{- \gamma / 2} e^{\xi_t}$ and using the above relation, and
after a tedious calculation, we get
\begin{align}
\phi_t^{- 1}  \tilde{H}' = & - | \nabla_{\Sigma_t} \xi_t |^2 -
\Delta_{\Sigma_t} \xi_t + (\tfrac{\gamma^2}{4} - \gamma) | \nabla_{\Sigma_t}
w|^2 - \tfrac{\gamma}{2} \Delta_{\Sigma_t} w - W_t - \tfrac{1}{n - 1}
\tilde{H}^2 - \tilde{H} q_t . \\
= : & L_t - W_t - \tfrac{1}{n - 1} \tilde{H}^2 - \tilde{H} q_t \\
\leqslant & L_t - \tilde{H} q_t,
\end{align}
where we have used $W_t \geqslant 0$ using the assumptions. We integrate
$\phi_t^{- 1}  \tilde{H}' \leqslant L_t - \tilde{H} q_t$ over $\Sigma_t$, and
we obtain
\begin{equation}
  \tilde{H}'  \int_{\Sigma_t} \phi_t^{- 1} + \tilde{H}  \int_{\Sigma_t} q_t
  \leqslant \int_{\Sigma_t} L_t \label{H tilde ODE}
\end{equation}
By the range $0 \leqslant \gamma < \tfrac{4}{n - 1}$ and the divergence
theorem,
\[ \int_{\Sigma_t} L_t = - \int_{\Sigma_t} (| \nabla_{\Sigma_t} \xi_t |^2 +
   (\gamma - \tfrac{\gamma^2}{4}) | \nabla_{\Sigma_t} w|^2) \leqslant 0. \]
By noting that $\tilde{H} (0) = 0$, and solving the inequality \eqref{H tilde
ODE}, we finish the proof.\end{proof}

Now we are ready to prove Theorem \ref{thm acg ric1}.

\begin{proof}[Proof of Theorem \ref{thm acg ric1}]
  Using Theorem \ref{thm non-exist strict barrier acg ric} and similar
  arguments as in Theorem \ref{thm acg}, we can show that $M$ is foliated by
  $\{\Sigma_t \}_{t \in [t_-, t_+]}$ hypersurfaces of vanishing $H + \gamma
  u^{- 1} u_{\nu} - \eta = 0$ for some $t_- < t_+$. It remains to calculate
  the rigid metric. To this end, we observe from Theorem \ref{thm non-exist
  strict barrier acg ric}, every leaf must satisfy the identities
\begin{align}
\nabla_{\Sigma_t} w & = 0, \\
|A|^2 - \tfrac{1}{n - 1} H^2 & = 0, \\
- \gamma u^{- 1} \Delta_g u + \operatorname{Ric} & = \Lambda,  \label{eq ode
ric1}\\
\operatorname{Ric} & = \operatorname{Ric} (\nu,\nu), \\
\langle w, \nu_{\Sigma_t} \rangle + \tfrac{n - 3}{2 (n - 1 - (n - 2)
\gamma)} \eta & = 0.  \label{eq ode w ric1}
\end{align}
  (Other than $\nabla_{\Sigma_t} w = 0$, the rest are implied by $Z = 0$ along
  $\Sigma_t$.) The condition $|A|^2 - \tfrac{1}{n - 1} H^2 = 0$ implies that
  $\Sigma_t$ is umbilic, hence a warped product $g = \mathrm{d} t^2 + \phi
  (t)^2 g_S$ for some closed manifold ($S, g_S$). We can assume that the $t$
  parametrizes the foliation as well. By $\nabla_{\Sigma_t} w=0$ and $w_{\nu} +
  \tfrac{n - 3}{2 (n - 1 - (n - 2) \gamma)} \eta = 0$, $u$ (since $w = \log
  u$) only depends on $t$. The equation \eqref{eq ode w ric1} is then an ODE
  for $u$ which we can solve, we obtain that $u = e^{\beta t}$. Now we solve
  \[ H + \gamma u^{- 1} u_{\nu} = \eta = (n - 1) \phi' / \phi + \gamma u^{- 1}
     u_{\nu} \]
  to get that $\phi = e^{\alpha t}$. The extra condition
  $\ensuremath{\operatorname{Ric}}_{g_S} \geqslant 0$ is the requirement that
  $\operatorname{Ric} (\partial_t, \partial_t$) is the least Ricci curvature
  required by $Z = W = 0$, see \eqref{eq W ric1} and Appendix \ref{warped
  product curvature}.
\end{proof}

\section{Band width estimates with spectral curvature bounds}\label{sec bw}

In this section, by selecting the suitable $h$ to be the composition of a
decreasing function $\eta$ and a distance function $\rho$, we prove band width
estimates (Theorems \ref{thm bw spec scalar}, \ref{bm bw} and \ref{thm bw
ric2}).

\subsection{Bonnet-Myers type band width estimate}

\begin{proof}[Proof of Theorem \ref{bm bw}]
  We show directly that if $\operatorname{width} (M, g) \geqslant t_+ - t_-$, then
  the width must be $t_+ - t_-$ and rigidity would follow. In particular, it
  would imply that $\ensuremath{\operatorname{width}} (M, g) \leqslant t_+ -
  t_-$.
  
  We set
  \[ \rho (x) = \min \{ \mathrm{dist}_g (x, \partial_- M) + t_+ - t_-, t_+ \},
  \]
  since the width is greater than $t_+ - t_-$, $\rho (x) = t_+$ for all $x \in
  \partial_+ M$. Also, $| \nabla \rho | \leqslant 1$. We set $h = \eta \circ
  \rho$. Then $H_{\partial_+ M} + \gamma u^{- 1} u_{\nu_+} \geqslant h$ on
  $\partial_+ M$, $H_{\partial_- M} + \gamma u^{- 1} u_{\nu_-} \leqslant h$ on
  $\partial_- M$. Also, we can easily check that $W \geqslant 0$ along every
  hypersurface $M$ from the estimate
  \[ h_{\nu} \geqslant \eta' \circ \rho \langle \nabla \rho, \nu \rangle
     \geqslant \eta' \circ \rho . \]
  With these conditions, we can prove as Theorem \ref{thm acg ric1} that $M$
  is foliated by hypersurfaces $\{\Sigma_t \}$ such that $H + \gamma u^{- 1}
  u_{\nu} - \eta \circ \rho = 0$ along $\Sigma_t$, we see that
  \[ \nabla_{\Sigma_t} w = 0, \text{ and } Z_t = 0 \text{ along } \Sigma_t .
  \]
  more specifically, we have the following identities,
\begin{align}
\nabla_{\Sigma_t} w & = 0, \\
|A|^2 - \tfrac{1}{n - 1} H^2 & = 0, \\
\langle \nabla \rho, \nu_{\Sigma_t} \rangle & = 1, \\
- \gamma u^{- 1} \Delta_g u + \operatorname{Ric} & = \Lambda, \label{eq ode ric1
eta} \\
\operatorname{Ric} & = \operatorname{Ric} (\nu,\nu), \\
\langle\nabla w, \nu_{\Sigma_t} \rangle + \tfrac{n - 3}{2 (n - 1 - (n - 2)
\gamma)} \eta \circ \rho & = 0 \label{eq ode w ric1 eta}
\end{align}
  along $\Sigma_t$ since all of the summands of $Z_{\Sigma_t}$ are
  non-negative. The condition $|A|^2 - \tfrac{1}{n - 1} H^2 = 0$ implies that
  $\Sigma_t$ is umbilic, hence a warped product $g = \mathrm{d} t^2 + \phi
  (t)^2 g_S$ for some closed manifold ($S, g_S$). With $\langle \nabla \rho,
  \nu_t \rangle = 1$, the level sets of $\rho$ agrees with the $t$-level set,
  and $\rho$ and $t$ differs by only a constant. By $\nabla_{\Sigma_t} w=0$ and
  $w_{\nu} + \tfrac{n - 3}{2 (n - 1 - (n - 2) \gamma)} \eta \circ \rho = 0$,
  $u$ (since $w = \log u$) only depends on $t$. The equation \eqref{eq ode w
  ric1 eta} is then an ODE for $u$ which we can solve, we obtain the
  expression of $u$. With this, \eqref{eq ode ric1 eta} is an ODE for
  $\phi$, which we can solve and it gives the model \eqref{eq ric1 model}.
  The extra condition is the requirement that $\operatorname{Ric} (\partial_t,
  \partial_t$) is the least Ricci curvature in all directions forced by $Z = W
  = 0$, see \eqref{eq W ric1} and Appendix \ref{warped product curvature}.
\end{proof}

\subsection{Band width estimate with spectral scalar curvature}

Now we briefly provide a proof of Theorem \ref{thm bw spec scalar} regarding
the band width estimates under spectral scalar curvature bounds. Most of the
proof was already in {\cite[Theorem 1.4]{chai-band-2025}}.

\begin{proof}[Proof of Theorem \ref{thm bw spec scalar}]
  We assume that $\ensuremath{\operatorname{width}} (M, g) \geqslant t_+ -
  t_-$, and set
  \[ \rho (x) = \min \{ \mathrm{dist}_g (x, \partial_- M) + t_+ - t_-, t_+ \},
  \]
  since the width is greater than $t_+ - t_-$, $\rho (x) = t_+$ for all $x \in
  \partial_+ M$. Also, $| \nabla \rho | \leqslant 1$. We set $h = \eta \circ
  \phi$. As in the proof of Theorem \ref{bm bw}, $M$ is foliated by
  hypersurfaces with vanishing $\tilde{H}$. We can use the rigidity analysis
  of {\cite[Theorem 1.4]{chai-band-2025}} to conclude the proof.
\end{proof}

\begin{remark}
  \label{rk omit proof ric2}Similarly, we can reuse the proof of Theorem
  {\cite[Theorem 1.2]{chai-band-2025}} to give the proof of Theorem \ref{thm
  bw ric2}, so we omit the proof of Theorem \ref{thm bw ric2}.
\end{remark}

\section{Spectral splitting and non-compact settings}\label{sec noncompact}

In this section, we show some applications of the band width estimates for the
splitting theorems of the spectral curvatures.

\subsection{Proof of splitting result under spectral scalar curvature bound}
In this subsection, we prove Theorem \ref{thm global splitting scalar}. We only need to prove $u$ is constant on $M$. We need some lemmas that can be found in \cite{zhu-rigidity-2020}.
\begin{lemma}[{\cite[Lemma 2.1]{zhu-rigidity-2020}}]
  \label{Lem-function-phi}There is a proper and surjective smooth function
  $\phi : M \to \mathbb{R}$ such that $\phi^{- 1} (0) = \Sigma$, 
  $\mathrm{Lip} \phi \leqslant 1$, and $\Sigma \subset M$ is an orientable closed
  hypersurface associated with a signed distance function.
\end{lemma}

For the function $\phi$, let
\[ \Omega_0 =\{x \in M: \text{ } \phi (x) < 0\}. \]
Given any smooth function $h : (- T, T) \to \mathbb{R}$, we introduce the
following functional
\[ \mathcal{B}^h (\Omega) = \int_{\partial^{\ast} \Omega} u^{\gamma}
   d\mathcal{H}^{n - 1} - \int_M (\chi_{\Omega} - \chi_{\Omega_0}) u^{\gamma}
   h \circ \phi \hspace{0.17em} d\mathcal{H}^n, \]
on
\[ \mathcal{C}_T =\{ \text{Caccioppoli set } \Omega \subset M :\text{ }\Omega \Delta
   \Omega_0 \Subset \phi^{- 1} ((- T, T))\}. \]
For the minimizing problem of the functional $\mathcal{B}^h$ on
$\mathcal{C}_T$, we have the following existence result.

\begin{lemma}[{\cite[Lemma 2.2]{zhu-rigidity-2020}}]
  Assume that $\pm T$ are regular values of $\phi$ and also the function $h$
  satisfies
  \begin{equation}
    \lim_{t \to - T} h (t) = + \infty \quad \text{ and } \quad \lim_{t \to T} h
    (t) = - \infty,
  \end{equation}
  then there exists a smooth minimizer $\hat{\Omega}$ in $\mathcal{C}_T$ for
  $\mathcal{B}^h$.
\end{lemma}

\begin{lemma}[cf. {\cite[Lemma 2.3]{zhu-rigidity-2020}}]
  \label{Lem-function-h}For any $\epsilon \in (0, 1)$, there is a function
  \[ h_{\epsilon} : \left( - \frac{1}{n \epsilon}, \frac{1}{n \epsilon}
     \right) \to \mathbb{R} \]
  such that
\begin{enumerate}[(a)]
    \item $h_{\epsilon}$ satisfies
    \[ \frac{2n-\gamma n+\gamma}{2(n - 1)+\gamma(2-n)} h^2_{\epsilon} +  2h_{\epsilon}' = \frac{2(n-1)+\gamma(2-n)}{2n+\gamma-\gamma n}n^2 \epsilon^2
     \]
     on \( \left( - \frac{1}{n \epsilon}, - \frac{1}{2 n}
       \right] \cup \left[ \frac{1}{2 n}, \frac{1}{n \epsilon} \right)\)
    and there is a universal constant $C$ so that
    \[ \sup_{- \frac{1}{2 n} \leqslant t \leqslant\frac{1}{2 n}} \left| \frac{2n-\gamma n+\gamma}{2(n - 1)+\gamma(2-n)} h^2_{\epsilon} +  2h_{\epsilon}' \right| \leqslant C \epsilon . \]
    \item $h'_{\epsilon} < 0$ and
    \[ \lim_{t \to \mp \frac{1}{n \epsilon}} h_{\epsilon} (t) = \pm \infty .
    \]
    \item As $\epsilon \to 0$, $h_{\epsilon}$ converge smoothly to $0$ on any
    closed interval.
    
    \item $h_{\epsilon} < 0$ on $\left[ \frac{1}{2 n}, \frac{1}{n \epsilon}
    \right)$.
  \end{enumerate}
\end{lemma}

Let
\[ \mathcal{B}_{\epsilon} (\Omega) = \int_{\partial^{\ast} \Omega} u^{\gamma}
   d\mathcal{H}^{n - 1} - \int_M (\chi_{\Omega} - \chi_{\Omega_0}) u^{\gamma}
   h_{\epsilon} \circ \phi \hspace{0.17em} d\mathcal{H}^n, \]
on
\[ \mathcal{C}_{\epsilon} =\{ \text{Caccioppoli set } \Omega \subset M :
   \Omega \Delta \Omega_0 \Subset \phi^{- 1} \left( \left( - \frac{1}{n
   \epsilon}, \frac{1}{n \epsilon} \right) \right) \}. \]
\begin{proposition}
  For almost every $\epsilon \in (0, 1)$, there is a smooth minimizer
  $\hat{\Omega}_{\epsilon}$ in $\mathcal{C}_{\epsilon}$ for the functional
  $\mathcal{B}^{\epsilon}$.
\end{proposition}
\begin{proof}[Proof of Theorem \ref{thm global splitting scalar}]
Similar to \cite[Section 4.2]{chai-band-2025}, we can obtain a closed stable weighted minimal surface $\Sigma$ by letting $\epsilon\to 0$. By Lemma \ref{lm rewrite prelim} for $h=0$,
\begin{align}
& \int_{\Sigma} (\tfrac{4}{4 - \gamma} | \nabla_{\Sigma} \psi |^2 +
\tfrac{1}{2} R_{\Sigma} \psi^2) \\\geqslant&\int_{\Sigma} (1 - \tfrac{\gamma}{4}) \gamma \left| \psi \nabla_{\Sigma} w
- \tfrac{1}{2 (1 - \gamma / 4)} \nabla_{\Sigma} \psi \right|^2\\
&+\int_{\Sigma}
(\tfrac{n}{2 (n - 1)} \gamma^2 - \gamma^2 + \gamma) w_{\nu}^2 \psi^2 + (-
\gamma \tfrac{\Delta u}{u} + \tfrac{1}{2} R_g) \psi^2 .
\end{align}
Since $\gamma <  \tfrac{2n}{n - 1}$, we have
\[  \tfrac{4 (n - 2)}{n -
   3} > \tfrac{8}{4 - \gamma} \text{ and } \tfrac{n}{2 (n - 1)}
   \gamma^2 - \gamma^2 + \gamma > 0. \]
If $u$ is not a constant on $\Sigma$ or $w_{\nu}\neq 0$ or $-
\gamma \tfrac{\Delta u}{u} + \tfrac{1}{2} R_g>0$, then, for $4\leqslant n\leqslant 7$,  the operator 
\[-\frac{4(n-2)}{n-3}\Delta_{\Sigma}+R_{\Sigma}\]
is positive, which is a contradiction, since $\Sigma$ is homologous to $T^{n-1}$; for $n=3$, the operator
\[-\frac{4}{4-\gamma}\Delta_{\Sigma}+\frac{1}{2}R_{\Sigma}\]
is positive, which is also a contradiction, since $\Sigma$ is homologous to $T^{2}$. Therefore, we have $u$ is a constant on $\Sigma$ and $w_{\nu}= 0$ and $-
\gamma \tfrac{\Delta u}{u} + \tfrac{1}{2} R_g=0$. Then, by a foliation argument (see \cite[Section 4.3]{chai-band-2025}), we can show that $u$ is a constant on the whole $M$. And then, $R_{g}\geq0$. The remaining proof is the same as \cite{zhu-rigidity-2020}.
\end{proof}

\subsection{Geroch conjecture with arbitrary ends under spectral
scalar curvature condition}
\begin{proof}[Proof of Theorem \ref{thm spectral geroch arbitrary end}]
We follow the strategy of {\cite[Theorem 3]{chodosh-generalized-2024}}. Because of the
positivity of the spectral scalar curvature, there is some room so that we can
select an $h$. The level sets of $h$ where it takes the values of $\pm \infty$
serve as barriers for the existence of warped $h$-bubbles. The only
differences are the construction of $h$ and the use of warped $\mu$-bubbles. We first introduce some notations as was done in {\cite{chodosh-generalized-2024}}.

Fix $\varepsilon > 0$ small, we define
\[ \Xi := \{x = (x_1, \ldots, x_n) \in \mathbb{R}^n : |x - k| > \varepsilon, k
   \in \mathbb{Z}^n \}/ \sim \]
where ($x_1, \ldots, x_n) \sim (x_1 + k_1, \ldots, x_n + k_n$) for $k = (k_1,
\ldots, k_n) \in \mathbb{Z}^n$. By assumption, there is a map $\Psi : \Xi \to
M$ so that $\Psi$ is a diffeomorphism onto its image. By scaling, we can
assume that $\Lambda > 1$ on $\Psi (\Xi$).

Observe that $M$ is (topologically) covered by $\hat{M} = (T^{n-1} \times \mathbb{R}) \#_{\mathbb{Z}} X$. (Unwind one of the $S^1$ factors in $T^n$.) Define
\[
\hat{\Xi} = \left\{ \vec{x} = (x_1, \dots, x_n) \in \mathbb{R}^n : |\vec{x} - \vec{k}| > \varepsilon, \vec{k} \in \mathbb{Z}^n \right\} \big/ \sim
\]
where $(x_1, \dots, x_n) \sim (x_1 + k_1, \dots, x_{n-1} + k_{n-1}, x_n)$ and note that the map $\Psi$ lifts to $\hat{\Psi} : \hat{\Xi} \to \hat{X}$, a diffeomorphism onto its image $\hat{M}_0$. It is useful to write
\[
\hat{M} = \hat{M}_0 \cup \left( \bigcup_{k \in \mathbb{Z}} \mathring{X}_k \right)
\]
where each $\mathring{X}_k$ is (topologically $X \setminus B$ for an $n$-ball $B$ in $X$) attached to $\hat{M}_0 := \hat{\Psi}(\hat{\Xi})$ along small spheres centered at $(0, 0, k)$.

Define $\gamma_1 = \sqrt{\frac{- n \gamma + \gamma + 2 n}{4 (n - 1) + 2
\gamma (2 - n)}}$.

The $\rho_0$ and $\rho_1$ are defined in the same way as in
{\cite{chodosh-generalized-2024}}, we assume that $\mathrm{Lip} (\rho_1) < L$
and $L$ is taken so we can assume that $\tfrac{\pi}{2} \gamma_1^{- 1} L = J +
\tfrac{3}{4}$. On $\hat{M}_0 \cap \left\{ | \rho_1 | < \tfrac{\pi L}{2
\gamma_1} \right\}$, define
\[ h (p) = - \tfrac{1}{\gamma_1} \tan (\tfrac{\gamma_1}{L} \rho_1 (p)) . \]
For $0 \leqslant k \leqslant J$ and
\[ p \in X^0_k \cap \{\rho_1 < k + \tfrac{1}{2} + \frac{L}{\gamma_1 \tan
   (\tfrac{\gamma_1}{L} (k + \tfrac{1}{2}))} \}, \]
or $- J \leqslant k < 0$ and
\[ p \in X^0_k \cap \{\rho_1 > k + \tfrac{1}{2} + \frac{L}{\gamma_1 \tan
   (\tfrac{\gamma_1}{L} (k + \tfrac{1}{2}))} \}, \]
we set
\[ h (p) = \frac{L}{\gamma_1^2  \left( \rho_1 - (k + \tfrac{1}{2}) -
   \tfrac{L}{\gamma_1 \tan (\tfrac{\gamma_1}{L} (k + \tfrac{1}{2}))} \right)}
   . \]
We can easily check as in {\cite[Lemma 22]{chodosh-generalized-2024}} that
\begin{equation}
  \gamma_1^2 h^2 - | \nabla h| + \Lambda \geqslant \gamma_1^2 h^2 - | \nabla
  h| + 1 > 0 \label{eq min stay compact}
\end{equation}
on $\{|h| < \infty\}$. We can smooth $h$ slightly, so that the above is still
satisfied. We still denote by $h$.

We fix
\[ \Omega_0 := \left( \hat{\Psi} (\hat{\Xi} \cap \{x_n < - \tfrac{1}{2} \}
   \cup (\cup_{k < 0} X_k^0)) \right) \cap \{|h| < \infty\}. \]
We minimize
\[ E (\Omega) = \int_{\Omega} u^{\gamma} + \int_M (\chi_{\Omega} -
   \chi_{\Omega_0}) hu^{\gamma} \]
for all Caccioppoli sets $\Omega$ in $M$ with $\Omega \Delta \Omega_0$ relatively compact in $M$.
Denote by $\Omega$ the connected component of the minimizer containing
$\{\rho_1 = - J\}$. By the stability {\cite[(4.5)]{chai-band-2025}} and
\eqref{eq min stay compact}
\[ \tfrac{4}{4 - \gamma} \int_{\partial \Omega} | \nabla_{\partial \Omega}
   \psi |^2 + \tfrac{1}{2} \int_{\partial \Omega} R_{\partial \Omega} \psi^2 >
   \int_{\partial \Omega} [\gamma_1^2 h^2 - | \nabla h| + \Lambda] \psi^2 > 0
\]
for all $\psi \in C^{\infty} (\Sigma)$. The rest of the argument is the same as
{\cite{chodosh-generalized-2024}}.
\end{proof}

We now prove Theorem \ref{thm cylinder
splitting} by adapting the argument of {\cite{chodosh-splitting-2019}}.

\begin{proof}[Proof of Theorem \ref{thm cylinder splitting}]
  It follows from Theorem \ref{thm weighted fcs} that $\Sigma$ is flat. By
  scaling if needed, we can assume that $\Sigma$ is isometric to the standard
  cylinder $\mathbb{S}^1 \times \mathbb{R}$. If $\Sigma$ is separating, then
  $M\setminus \Sigma$ has two components; we choose one and denote it by
  $(\hat{M}, \hat{g})$. If $\Sigma$ is non-separating, then we cut
  along $\Sigma$ to obtain a new manifold which we denote also by $(\hat{M},
  \hat{g})$; $(\hat{M}, \hat{g})$ has two boundary components, from which we
  choose one. We set $\mathbb{S}^1 \times \{0\} \subset \Sigma$ to be $\ell$
  and $\Sigma_h$ to be $\mathbb{S}^1 \times [- h, h] \subset \Sigma$.
  
  Fix a unit speed geodesic $c : [0, \varepsilon) \to \hat{M}$ with $c (0) \in
  \ell$ and the tangent vector at $c (0)$ is normal to $\Sigma$. We can
  construct a family of positive functions $u_{r, t}$ (see {\cite[Lemma
  3.3]{hong-splitting-2025}}, cf. {\cite[Lemma 2.2]{antonelli-sharp-2024}})
  with the following conditions:
\begin{enumerate}[(a)]
    \item $u_{r, t} \to u$ in $C^3$ as $t, r \to 0$;
    
    \item $u_{r, t} \to u$ smoothly as $t \to 0$ for $r \in (0, \varepsilon)$
    fixed;
    
    \item $u_{r, t} = u$ on $\{x \in \hat{M} : \text{ }
    \ensuremath{\operatorname{dist}}_{\hat{g}} (x, c (2 r)) \geqslant 3 r\}$;
    
    \item $u_{r, t} < u$ on $\{x \in \hat{M} : \text{ }
    \ensuremath{\operatorname{dist}}_{\hat{g}} (x, c (2 r)) < 3 r\}$;
    
    \item $- \gamma u_{r, t}^{- 1} \Delta_g u_{r, t} + \tfrac{1}{2} R_g > 0$
    on $\{x \in \hat{M} : \text{ } r
    <\ensuremath{\operatorname{dist}}_{\hat{g}} (x, c (2 r)) < 3 r\}$;
    
    \item $H_{\Sigma} + \gamma u_{r, t}^{- 1} \partial_{\nu} u_{r, t}
    \geqslant 0$ (which we call weakly $u_{r, t}^{\gamma}$-weighted
    mean-convex with respect to the weight $u_{r, t}^{\gamma}$).
  \end{enumerate}
  Fix $h > 1$. Let $B_h$ denote a precompact open set with smooth boundary in
  $\hat{M}$ and such that $\{x \in \hat{M} : \text{ }
  \ensuremath{\operatorname{dist}}_{\hat{g}} (x, \Sigma_h) < 2 h\} \subset
  B_h$. We modify $u_{r, t}$ further near the boundary of $B_h$ to $u_{r, t,
  h}$ so $B_h$ is weakly weighted mean-convex with respect to the weight
  $u_{r, t, h}^{\gamma}$ and
  \begin{equation}
    (1 - \delta) u_{r, t} \leqslant u_{r, t, h} \leqslant (1 + \delta) u_{r,
    t} \label{u r t h}
  \end{equation}
  where $\delta \in (0, 1)$ is chosen to satisfy the relation \eqref{choice of
  delta}. Among all compact, oriented surfaces in $B_h$ with respect to
  $\partial \Sigma_h$ that bound an open subset of $\hat{M}$, there is one
  whose weighted area with respect to $u_{r, t, h}$ is least. Choose one such
  weight area-minimizing surface and denote it by $\Sigma_{r, t, h}$.
  
  We claim that $\Sigma_{r, t, h}$ intersects $\{x \in \hat{M} : \text{ }
  \ensuremath{\operatorname{dist}}_{\hat{g}} (x, c (2 r)) < 3 r\}$. For if
  not, we have that $u_{r, t, h} = u$ along $\Sigma_{r, t, h}$, and
\begin{align}
0 < & \mathcal{A}_u (\Sigma_h) -\mathcal{A}_{u_{r, t}} (\Sigma_h)
\label{compare 1} \\
\leqslant & \mathcal{A}_u (\Sigma_{r, t, h}) -\mathcal{A}_{u_{r, t}}
(\Sigma_h) \label{compare 2} \\
= & \mathcal{A}_{u_{r, t}} (\Sigma_{r, t, h}) -\mathcal{A}_{u_{r, t}}
(\Sigma_h) \label{compare 3} \\
\leqslant & \tfrac{1}{1 - \delta} \mathcal{A}_{u_{r, t, h}} (\Sigma_{r, t,
h}) -\mathcal{A}_{u_{r, t}} (\Sigma_h) \label{compare 4} \\
\leqslant & \tfrac{1}{1 - \delta} \mathcal{A}_{u_{r, t, h}} (\Sigma_h)
-\mathcal{A}_{u_{r, t}} (\Sigma_h) \label{compare 5} \\
\leqslant & \tfrac{1 + \delta}{1 - \delta} \mathcal{A}_{u_{r, t}}
(\Sigma_h) -\mathcal{A}_{u_{r, t}} (\Sigma_h) \label{compare 6}
\\
= & \tfrac{2 \delta}{1 - \delta} \mathcal{A}_{u_{r, t}} (\Sigma_h) .
\label{compare 7}
\end{align}
  We explain these relations: \eqref{compare 1} follows since
  $\{\ensuremath{\operatorname{dist}}_{\hat{g}} (x, c (2 r)) < 3 r\} \cap
  \Sigma_h$ is non-empty and $u_{r, t} < u$ on this set; \eqref{compare 2}
  follows since $\Sigma_h$ minimizes the $u^{\gamma}$-weighted area;
  \eqref{compare 3} follows since $\Sigma_{r, t, h}$ does not intersect
  $\{\ensuremath{\operatorname{dist}}_{\hat{g}} (x, 2 r) < 3 r\}$ on which
  $u_{r, t}$ and $u$ differ; \eqref{compare 4} follows from \eqref{u r t h};
  \eqref{compare 5} holds since $\Sigma_{r, t, h}$ minimizes $u_{r, t,
  h}^{\gamma}$-weighted area; again, \eqref{compare 6} is due to \eqref{u r t
  h}.
  
  This is impossible if we choose $\delta = \delta (r, t, h) > 0$ with
  \begin{equation}
    \tfrac{2 \delta}{1 - \delta} < \tfrac{\mathcal{A}_u (\Sigma_h)
    -\mathcal{A}_{u_{r, t}} (\Sigma_h)}{\mathcal{A}_{u_{r, t}} (\Sigma_h)} .
    \label{choice of delta}
  \end{equation}
  Now we can take a subsequential limit as $h \to \infty$ and obtain a
  properly embedded surface $\Sigma_{r, t}$. From the construction,
  $\Sigma_{r, t}$ is a boundary and homologically* $u^{\gamma}_{r,
  t}$-weighted area-minimizing.
  
  Following {\cite{chodosh-splitting-2019}} with Theorem \ref{thm weighted
  fcs} in place of {\cite[Lemmas 2.1-2.4]{chodosh-splitting-2019}} as we take
  limit as $t \to 0$, we obtain a family of surfaces $\Sigma_r$ which
  converges to $\Sigma$. If $\Sigma_r$ is a torus, we again use Theorem
  \ref{thm weighted fcs} to show that $\hat{M}$ is isometric to either
  standard $\mathbb{S}^1 \times \mathbb{R} \times [0, \infty)$ or
  $\mathbb{S}^1 \times \mathbb{R} \times [0, a]$ for some $a > 0$. We may
  assume that $\Sigma_r$ is cylindrical. By the proof of Theorem \ref{thm
  weighted fcs}, the gradient of $u$ vanishes along $\Sigma_r$. Assume that
  $u$ take different values on $\Sigma$ and $\Sigma_{r_0}$ for some $r_0$
  sufficiently small, by connecting two points $x \in \Sigma$ and $y \in
  \Sigma_{r_0}$ by a segment $\ell$. There exists a point $z \in \ell$ not
  equal to $x$ nor $y$ such that $u$ has a non-vanishing gradient. Take a small
  geodesic ball $B_{r_1} (z)$ centered at $z$ such that $\nabla u$ is nowhere
  vanishing, the family $\{\Sigma_r \}$ cannot intersect $B_{r_1} (z)$ from
  which we obtain a contradiction with that $\{\Sigma_r \}$ converges to
  $\Sigma$. Hence, $u$ is constant along all $\Sigma_r$ from which we are
  reduced to the case $u$ is constant, the case handled by
  {\cite{chodosh-splitting-2019}} and we finish the proof.
\end{proof}

\begin{remark}
  It is direct to define the notions of absolutely $u^{\gamma}$-weighted
  area-minimizing, homologically $u^{\gamma}$-weighted area-minimizing and
  homologically* $u^{\gamma}$-weighted area minimizing by adapting
  {\cite[Appendix]{chodosh-splitting-2019}} to the weighted case.
\end{remark}

\subsection{Proof of Theorem \ref{thm-spectral-classify-results}}

This subsection is devoted to the proof of \ref{thm-spectral-classify-results}.
We only need to prove $u$ is a constant on $M$. Then we can use the result of Liu \cite{Liu-three-Invent}.

First, by the result of Antonelli-Pozzetta-Xu {\cite[Theorem
1.1]{antonelli-sharp-2024}} and its proof, we may assume $\pi_2 (M) = 0$ and $M$ is not
diffeomorphic to $\mathbb{R}^3$. Next, by passing to a suitable covering, we
can further assume $\pi_1 (M) =\mathbb{Z}$ and $M$ is orientable. Let $\Gamma$
represent the generator of the fundamental group of $M$. We may assume
$\Gamma$ is a smooth closed curve. Consider an exhaustion of $M$ by $\Omega_i$
with smooth boundary $\partial \Omega_i$, where we can assume $\Gamma$ lies in
each $\Omega_i$. By Poincar{\'e} duality for manifolds with boundary, there
exists an oriented surface $\Sigma_i \subset \Omega_i$ such that $\partial
\Sigma_i \subset \Omega_i$ and the oriented intersection number of $\Sigma_i$
with $\Gamma$ equals $1$. We then consider minimizing the weighted area
$\int_{\Sigma} u^{\gamma}$ over all surfaces that belong to the same homology
class as $\Sigma_i$ and have the same boundary as $\Sigma_i$. We can perturb
the metric near $\partial \Omega_i$ such that $H+\gamma u^{-1}u_{\nu}>0$ on the boundary $\partial \Omega_{i}$. For each $i$, there exists a weighted minimizing
surface, which we still denote as $\Sigma_i$, and the intersection of
$\Sigma_i$ with $\Gamma$ is nonempty. By the curvature estimate \cite[Theorem
3.6]{zhou-existence-2020}, a subsequence of $\Sigma_i$ converges to
an oriented stable weighted minimal surface $\Sigma$ in $M$. By Lemma \ref{lm rewrite prelim} with $h=0$ and $H+ \gamma u^{-1}u_{\nu}=0$, we have
\begin{align*}
    \int_{\Sigma}\left[\frac{4}{4-\gamma}|\nabla_{\Sigma}\psi|^2+R_{\Sigma}\psi^2\right]&\geqslant\int_{\Sigma} (1 - \tfrac{\gamma}{4}) \gamma \left| \psi \nabla_{\Sigma} w
- \tfrac{1}{2 (1 - \gamma / 4)} \nabla_{\Sigma} \psi \right|^2+\int_{\Sigma} w_{\nu}^2\psi^2\\
&\quad+\int_{\Sigma}\left(-\gamma u^{- 1} \Delta_g u +2\operatorname{Ric}_{g}\right)\psi^2,
\end{align*}
where $w = \log u$ and $\psi\in C^{\infty}_{c}(\Sigma)$.

If $-\gamma u^{- 1} \Delta_g u +\operatorname{Ric}_{g}> 0$ and $-\gamma u^{- 1} \Delta_g u +2\operatorname{Ric}_{g}> 0$, then $\Sigma$ is either compact and diffeomorphic to $\mathbb{S}^2$( which contradicts $\pi_{2}(M)=0$) or $\Sigma$ is non-compact. 
When $\Sigma$ is non-compact, there are two subcases: 
(i) $\Sigma$ is conformal to the cylinder, which contradicts $(b)$ in Theorem \ref{thm weighted fcs};
(ii) $\Sigma$ is conformal to the complex plane.

In Case (ii), we apply Lemma \ref{lm rewrite prelim} again with $h=0$ and $H+\gamma u^{-1}u_{\nu}=0$, yielding
\begin{align}\label{ineqn for spec ricci}
 \tfrac{4}{4 - \gamma}\int_{\Sigma} | \nabla_{\Sigma} \psi |^2& \geq
  \int_{\Sigma} (1 - \tfrac{\gamma}{4}) \gamma \left| \psi \nabla_{\Sigma} w
- \tfrac{1}{2 (1 - \gamma / 4)} \nabla_{\Sigma} \psi \right|^2 \\
& \quad + \int_{\Sigma}\left(-\gamma u^{- 1} \Delta_g u +\operatorname{Ric}_{g}\right)  \psi^2+ \int_{\Sigma}(\gamma-\tfrac{\gamma^2}{2}) w_{\nu}^2\psi^2,
\end{align}
  where $w = \log u$ and $\psi\in C^{\infty}_{c}(\Sigma)$. Recall that we assume that $-\gamma u^{- 1} \Delta_g u +\operatorname{Ric}_{g}\geqslant0$ and $0<\gamma<2$, by \cite[Theorem 1.1]{berard-inverse-2014}, $\Sigma$ has at most quadratic volume growth. We can then apply the logarithmic cut-off trick (see \cite[Proposition 1.37]{CM-2011-book}) to deduce that $-\gamma u^{- 1} \Delta_g u +\operatorname{Ric}_{g}=0$, $(\log u)_{\nu}=0$ and $\nabla_{\Sigma}\log u=0$. This is a contradiction. 

The remaining case is $-\gamma u^{- 1} \Delta_g u +\operatorname{Ric}_{g}\geqslant0$ or $-\gamma u^{- 1} \Delta_g u +2\operatorname{Ric}_{g}\geqslant0$.  Following Liu {\cite{Liu-three-Invent}}, we handle this case as follows. Fixed a
point $p \in M$ such that $p \not\in \Gamma$; we aim to deform the metric $g$ so that the spectral Ricci curvature is strictly positive in an annulus
region around $p$. Let $g_t = e^{2 tf} g$ and $|v|_g = 1$, we then have

\begin{align*}
  & \quad - \gamma \frac{\Delta_{g_t} u}{u} + \operatorname{Ric}_{g_t}\\
  & = e^{- 2 tf}  \left[ - \frac{\Delta_g u}{u} + \operatorname{Ric}_g \right]\\
  & \quad + e^{- 2 tf}  \left[ \frac{t}{u}  \hspace{0.27em} g (\nabla u,
  \nabla f) - t \nabla^2 f (v, v) - t \Delta_g f + t^2  ((v (f))^2 - | \nabla
  f|^2) \right]
\end{align*}

where we have used the identity
\[ \Delta_{g_t} u = e^{- 2 tf}  [\Delta_g u + t \langle \nabla u, \nabla f
   \rangle_g] . \]
Let $r$ denote the distance function to $p$. For a sufficiently small $R > 0$,
consider the function $\rho = R - r$ defined for $\frac{R}{2} < r < R$; we
then extend $\rho$ to be a positive smooth function for $0 \leqslant r <
\frac{R}{2}$. With this $\rho$, define $f = - \rho^5$. For $\frac{R}{2} < r <
R$, we obtain

\begin{align*}
  & \quad - \gamma \frac{\Delta_{g_t} u}{u} + \operatorname{Ric}_{g_t}\\
  & \geqslant e^{- 2 tf}  \left[ - \frac{\Delta_g u}{u} + \operatorname{Ric}_g \right]\\
  & \quad + e^{- 2 tf}  \left[ 5  t \rho^4u^{-1} g (\nabla u,
  \nabla \rho)+ 20 t \rho^3 + 5
  t \rho^4 (\Delta_g \rho + \nabla^2 \rho (v, v)) - 25 t^3 \rho^8 \right] .
\end{align*}

By the almost Euclidean propery of $M$ near $p$, for small $R$, we have
\[ | \Delta_g \rho + \nabla^2 \rho (v, v) | \leqslant\frac{9}{8 (R - \rho)} . \]
For all small $t$, if $r$ sufficiently close to $R$, then $\rho$ is close to
$0$. In this case,
\[ 20 t \rho^3 - 5  t \rho^4u^{-1} g (\nabla u,
  \nabla \rho) + 5 t \rho^4 (\Delta_g \rho +
   \nabla^2 \rho (v, v)) - 25 t^3 \rho^8 > 0 \]
since the leading term is $\rho^3$ when $\rho$ is small. When $t$ small enough, we can also ensure $- \gamma \frac{\Delta_{g_t} u}{u} + 2\operatorname{Ric}_{g_t}>0$ at the same time. Note that the metric
$g_t$ remains unchanged outside $B_p (R)$, and this metric deformation(i.e.,
$g_t$) is $C^4$ continuous with respect to the metric $g$ and
$C^{\infty}$-continuous with respect to $t$.

Since $\Gamma$ is closed, we can apply this perturbation finitely many times
such that the spectral Ricci curvature is positive on $\Gamma$ (each time we
perturb the metric a little bit around a point) and nonnegative except for a
small neighborhood of $p$. We then minimize the weighted area functional as we
did earlier, which yields a complete stable weighted minimal surface $\Sigma$.

We now claim that $\Sigma$ must pass through this small neighborhood of $p$.
If this were not the case, the spectral Ricci curvature on $\Sigma$ would be
nonnegative, and strictly positive at some point on $\Gamma$. This leads to a
contradiction as before.

Let $t$ denote the deformation parameter. We shrink the size of the
neighborhood of $p$ where the spectral Ricci curvature might be negative; this
allows us to construct a sequence of metrics on $M$, each admitting a stable
weighted minimal surface passing through the small neighborhood of $p$. We can
choose $t \to 0$ sufficiently rapidly such that these metrics converge to the
original metric in a $C^4$ sense.

By passing to a subsequence of these complete weighted minimal surfaces and
taking the limit, we obtain a completely oriented stable weighted minimal
surface $\Sigma$ passing through $p$ with respect to the original metric. This holds for any $p\in M\setminus \Gamma$.

By Lemma \ref{lm rewrite prelim}, for $h=0$ and $H+\gamma u^{-1}u_{\nu}=0$, we have
\begin{align*}
    \int_{\Sigma}\left[\frac{4}{4-\gamma}|\nabla_{\Sigma}\psi|^2+R_{\Sigma}\psi^2\right]&\geqslant\int_{\Sigma} (1 - \tfrac{\gamma}{4}) \gamma \left| \psi \nabla_{\Sigma} w
- \tfrac{1}{2 (1 - \gamma / 4)} \nabla_{\Sigma} \psi \right|^2+\int_{\Sigma} w_{\nu}^2\psi^2\\
&\quad+\int_{\Sigma}\left(-\gamma u^{- 1} \Delta_g u +2\operatorname{Ric}_{g}\right)\psi^2,
\end{align*}
and 
\begin{align}\label{ineqn for spec ricci}
 \tfrac{4}{4 - \gamma}\int_{\Sigma} | \nabla_{\Sigma} \psi |^2& \geq
  \int_{\Sigma} (1 - \tfrac{\gamma}{4}) \gamma \left| \psi \nabla_{\Sigma} w
- \tfrac{1}{2 (1 - \gamma / 4)} \nabla_{\Sigma} \psi \right|^2 \\
& \quad + \int_{\Sigma} \left(-\gamma u^{- 1} \Delta_g u +\operatorname{Ric}_{g}\right)  \psi^2+ \int_{\Sigma} (\gamma-\tfrac{\gamma^2}{2})w_{\nu}^2\psi^2,
\end{align}
where $w = \log u$ and $\psi\in C^{\infty}_{c}(\Sigma)$.

If $u$ is not constant on $\Sigma$, then using the same argument as before (given our assumptions that $-\gamma u^{- 1} \Delta_g u +\operatorname{Ric}_{g}\geqslant0$, $-\gamma u^{- 1} \Delta_g u +2\operatorname{Ric}_{g}\geqslant0$ and $0<\gamma<2$),  we can deduce that $-\gamma u^{- 1} \Delta_g u +\operatorname{Ric}_{g}=0$, $(\log u)_{\nu}=0$ and $\nabla_{\Sigma}\log u=0$, i.e., $u$ is a constant on $\Sigma$. This is a contradiction. 

  This implies that $\operatorname{Ric}_{g}$ is non-negative on $p\in M\setminus \Gamma$ and that $u$ is constant on $M\setminus \Gamma$. By the continuity of $\operatorname{Ric}_{g}$, we conclude that $\operatorname{Ric}_{g}$ is non-negative on the entire manifold $M$. Then we can use the result of Liu \cite{Liu-three-Invent}.

\appendix\section{\label{warped product curvature}Warped product curvatures}

In this appendix, we calculate the frequently used curvatures and relations of
a warped product and a doubly warped product.

\subsection{Warped product}\label{app warp} Let $g = \mathrm{d} t^2 + \phi^2 (t)
\bar{g}$ where $g$ is a metric on the manifold $M$. We calculate $\operatorname{Ric}
(e_i, e_i$) where $e_i$ is tangential to $M$ of unit $g$-length. For
convenience, we set $a_i = \overline{\operatorname{Ric}} (\phi e_i, \phi e_i$) where
$\overline{\operatorname{Ric}}$ is the Ricci curvature of ($M, \bar{g}$). Note that
$\phi e_i$ is of unit $\bar{g}$-length and $\{a_i \}$ are the Ricci curvature
of the metric $\bar{g}$.

By the Gauss equation,
\begin{align}
R_{ii} = & \sum_j R_{ijji} + R_{i \nu \nu i} \\
& = \sum_j (\bar{R}_{ijji} - h_{jj} h_{ii} + h_{ij} h_{ij}) + R_{i \nu \nu
i} \\
& = \bar{R}_{ii} - Hh_{ii} + h_{ii}^2 + R_{i \nu \nu i} \\
& = a_i \phi^{- 2} - (n - 2) (\phi' / \phi)^2 + R_{i \nu \nu i}
\end{align}
where we used that $H = (n - 1) \tfrac{\phi'}{\phi}$ and $h_{ii} = \phi' /
\phi$, it remains to calculate $\bar{R}_{i \nu \nu i}$. We first compute the
necessary Christofel symbols $\Gamma_{nn}^i$ and $\Gamma_{ni}^j$ of $g$:
\[ \Gamma_{nn}^i = \tfrac{1}{2} \sum_{l \leqslant n} g^{il} (2 g_{nl, n} -
   g_{nn, l}) = 0, \Gamma_{nn}^n = 0, \]
and
\begin{align}
\Gamma_{in}^j = \tfrac{1}{2} & \sum_{l \leqslant n} g^{jl} (g_{li, n} +
g_{ni, l} - g_{in, l}) = \tfrac{1}{2} \sum_{l \leqslant n - 1} \phi^{- 2}
\bar{g}^{jl} (\phi^2  \bar{g}_{li})_n = \phi' \delta_i^j / \phi .
\end{align}
And the component
\begin{align}
R_{i \nu \nu i} = & \partial_i \Gamma_{nn}^i - \partial_n \Gamma_{in}^i +
\sum_{s \leqslant n} (\Gamma_{nn}^s \Gamma_{is}^i - \Gamma_{in}^s
\Gamma_{ns}^i) \\
= & - (\phi' / \phi)' - \sum_{s \leqslant n - 1} \Gamma_{in}^s \Gamma_{ns}^i
= - (\phi' / \phi)' - (\phi' / \phi)^2
\end{align}
follows easily. Hence,
\[ R_{ii} = a_i \phi^{- 2} - (n - 1) (\phi' / \phi)^2 - (\phi' / \phi)' . \]
As for $\operatorname{Ric} (\nu, \nu$), we use the first variation of the mean
curvature of the $t$-level set. We get
\[ H' = - \operatorname{Ric} (\nu, \nu) - |A|^2 . \]
Hence,
\[ \operatorname{Ric} (\nu, \nu) = - (n - 1) (\phi' / \phi)' - (n - 1) (\phi' /
   \phi)^2 . \]
The scalar curvature $R_{g}$ of the metric $g$ is
\[R_{g}  = \frac{R_{\bar{g}}}{\phi^2}- 2 (n - 1) \frac{\phi''}{\phi} - (n - 1) (n - 2)
   \frac{(\phi')^2}{\phi^2}, \]
which can be found via Schoen-Yau's rewrite \eqref{eq sy rewrite}. The Laplace
of the metric is
\[ \Delta_g = \partial_t^2 + H_t \partial_t + \frac{1}{\phi^2}
   \Delta_{g_{\mathbb{T}^{n - 1}}} . \]
\subsection{Doubly warped product}\label{app doubly}Let $g = \mathrm{d} t^2 +
\phi (t)^2 \mathrm{d} s_1^2 + \varphi (t)^2 \mathrm{d} s_2^2$. The non-zero
components of its Ricci curvatures are given by
\begin{align}
\ensuremath{\operatorname{Ric}} (\partial_t, \partial_t) & = - (\phi
\varphi)^{- 1} (\phi \varphi'' + \varphi'' \phi), \\
\ensuremath{\operatorname{Ric}} (e_1, e_1) & = - (\phi \varphi)^{- 1}
(\varphi \phi'' + \varphi' \phi'), \\
\ensuremath{\operatorname{Ric}} (e_2, e_2) & = - (\phi \varphi)^{- 1}
(\varphi'' \phi + \varphi' \phi'),
\end{align}
where $e_1 = \phi^{- 1} \partial_{s_1}$ and $e_2 = \varphi^{- 1}
\partial_{s_2}$. It is convenient to introduce $G = \phi' / \phi + \varphi' /
\varphi$ and $F = \phi' / \phi - \varphi' / \varphi$. We often need to impose
that $\ensuremath{\operatorname{Ric}} (e_1, e_1)
=\ensuremath{\operatorname{Ric}} (e_2, e_2)$ and
$\ensuremath{\operatorname{Ric}} (\partial_t, \partial_t) \geqslant
\ensuremath{\operatorname{Ric}} (e_i, e_i)$. The former gives $\phi'' / \phi =
\varphi'' / \varphi$ and then
\[ F' = \tfrac{\phi''}{\phi} - (\tfrac{\phi'}{\phi})^2 -
   \tfrac{\varphi''}{\varphi} + (\tfrac{\varphi'}{\varphi})^2 = - F G. \]
That is,
\begin{equation}
  (\phi' / \phi - \varphi' / \varphi)' = - (\phi' / \phi - \varphi' / \varphi)
  (\phi' / \phi + \varphi' / \varphi) . \label{eq F G}
\end{equation}
With $\phi'' / \phi = \varphi' / \varphi$, the latter gives
\begin{align}
0 \geqslant & \tfrac{\phi''}{\phi} - \tfrac{\phi' \varphi'}{\phi \varphi}
\\
= & \tfrac{1}{2} (\tfrac{\phi''}{\phi} + \tfrac{\varphi''}{\varphi}) -
\tfrac{\phi' \varphi'}{\phi \varphi} \\
= & \tfrac{1}{2} ((\phi' / \phi + \varphi' / \varphi)' +
(\tfrac{\phi'}{\phi})^2 + (\tfrac{\varphi'}{\varphi})^2) - \tfrac{\phi'
\varphi'}{\phi \varphi} \\
= & \tfrac{1}{2} (G' + \tfrac{1}{2} (F^2 + G^2)) - \tfrac{1}{4} (G^2 - F^2)
\\
= & \tfrac{1}{2} (G' + F^2),
\end{align}
which is
\begin{equation}
  (\phi' / \phi + \varphi' / \varphi)' + (\phi' / \phi - \varphi' / \varphi)^2
  \leqslant 0 \label{eq G F}
\end{equation}

\bibliographystyle{alpha}
\bibliography{spec-rig}
\end{document}